\documentclass[a4paper,final]{article}
\usepackage{etex}
\usepackage[british]{babel}
\usepackage{amsmath, amssymb, amsthm, amsfonts}
\usepackage{ifdraft}
\usepackage{microtype}
\usepackage{url}
\DeclareMathAlphabet{\mathpzc}{OT1}{pzc}{m}{it}
\usepackage[utf8]{inputenc}
\usepackage{aliascnt}
\usepackage{xcolor}
\usepackage[colorlinks=true,linkcolor=blue,citecolor=green!60!black,draft=false]{hyperref}
\usepackage{mathtools}
\usepackage{enumerate}
\usepackage[all]{xy}
\usepackage{paralist}
\usepackage{subfigure}
\usepackage{multirow}
\usepackage{stmaryrd}
\usepackage{xargs}
\usepackage{tikz}
\usetikzlibrary{plotmarks}
\usepackage{pgfplots}
\usepackage{pgfplotstable}
\usepackage{booktabs}
\usepackage{colortbl}
\usepackage{arydshln}

\makeatletter
\newcommand{\mywarning}[1]{\@latex@warning{#1}}
\makeatother

\ifdraft{
\newcommand{\colorS}[1]{{\color{orange} #1}}
\newcommand{\colorV}[1]{{\color{red} #1}}
\newcommand{\itemS}[1]{\colorS{\item[S:] #1}}
\newcommand{\itemV}[1]{\colorV{\item[V:] #1}}
\newcommand{\commS}[1]{{\marginpar{\tiny \colorS{#1}}}}
\newcommand{\commV}[1]{{\marginpar{\tiny \colorV{#1}}}}
\newcommand{\comment}[1]{{\color{gray} #1}}

}
{
\newcommand{\colorS}[1]{#1}
\newcommand{\colorV}[1]{#1}
\newcommand{\itemS}[1]{}
\newcommand{\itemV}[1]{}
\newcommand{\commS}[1]{}
\newcommand{\commV}[1]{}
\newcommand{\comment}[1]{}

}

\makeatletter
\def\myfooter{\xdef\@thefnmark{}\@footnotetext}
\makeatother

\title{Constructing unlabelled lattices}
\author{\colorV{Volker Gebhardt}, \colorS{Stephen Tawn}}
\date{20 December 2019}

\hypersetup{pdftitle={Constructing unlabelled lattices},
            pdfauthor={Volker Gebhardt, Stephen Tawn},
            pdfkeywords={},
            pdfsubject={},
            pdfdisplaydoctitle={true}}

\newcommand{\NN}{\mathbb{N}}

\newcommand{\FF}{\mathbb{F}}

\newcommand{\upset}[2][]{{\uparrow}_{#1}\,#2}
\newcommand{\downset}[2][]{{\downarrow}_{#1}\,#2}
\newcommand{\isabove}[1][]{\mathrel{\sqsupseteq_{#1}}}
\newcommand{\isbelow}[1][]{\mathrel{\sqsubseteq_{#1}}}
\newcommand{\join}[1][]{\mathbin{\vee_{#1}}}
\newcommand{\meet}[1][]{\mathbin{\wedge_{#1}}}
\newcommand{\iscoveredby}[1][]{\mathrel{\prec_{#1}}}
\newcommand{\dep}[1][]{\mathrm{dep}_{#1}}
\newcommandx*{\cov}[2][1=,2=]{\mathchoice
                                {\bigcurlywedge{\rule[-1.2ex]{0pt}{1ex}}_{#1}^{#2}}
                                {\bigcurlywedge_{#1}^{#2}}
                                {\bigcurlywedge_{#1}^{#2}}
                                {\bigcurlywedge_{#1}^{#2}}}
\newcommandx*{\LWT}[2][1=,2=]{\mathcal{LW}_{#1}^{#2}}

\newcommand{\wt}{\mathrm{wt}}
\newcommand{\lev}{\mathrm{lev}}
\newcommand{\Stab}{\mathrm{Stab}}
\newcommand{\Sym}{\mathrm{Sym}}

\makeatletter
\pgfplotsset{
    /pgfplots/table/omit header/.style={%
        /pgfplots/table/typeset cell/.append code={%
            \ifnum\c@pgfplotstable@rowindex=-1
                \pgfkeyslet{/pgfplots/table/@cell content}\pgfutil@empty%
            \fi
        }
    }
}
\makeatother

\theoremstyle{plain}
\newtheorem{theorem}{Theorem}

\newaliascnt{lemma}{theorem}
\newtheorem{lemma}[lemma]{Lemma}
\aliascntresetthe{lemma}

\newaliascnt{proposition}{theorem}

\aliascntresetthe{proposition}

\newaliascnt{corollary}{theorem}
\newtheorem{corollary}[corollary]{Corollary}
\aliascntresetthe{corollary}

\newaliascnt{conjecture}{theorem}

\aliascntresetthe{conjecture}

\theoremstyle{definition}

\newaliascnt{definition}{theorem}
\newtheorem{definition}[definition]{Definition}
\aliascntresetthe{definition}

\newaliascnt{example}{theorem}

\aliascntresetthe{example}

\newaliascnt{notation}{theorem}
\newtheorem{notation}[notation]{Notation}
\aliascntresetthe{notation}

\newaliascnt{remark}{theorem}
\newtheorem{remark}[remark]{Remark}
\aliascntresetthe{remark}

\newaliascnt{claim}{theorem}

\aliascntresetthe{claim}

\newtheorem*{claim*}{Claim}

\addto\extrasbritish{

}

\makeatletter
\def\MR@url=#1 #2={http://www.ams.org/mathscinet-getitem?mr=#1}
\newcommand{\MR}[1]{{\scriptsize\href{\MR@url =#1 =}{MR#1}}}
\makeatother

\newcommand{\arXiv}[1]{\href{http://arxiv.org/abs/#1}{\textsf{arXiv}:#1}}

\begin{document}

\maketitle
\myfooter{Both authors acknowledge support under UWS grant 20721.81112.
Volker Gebhardt acknowledges support under the Government of Spain Project MTM2013-44233-P.}
\myfooter{MSC-class: 05A15 (Primary) 06A07, 05-04 (Secondary)}
\myfooter{Keywords: unlabelled lattice, enumeration, orderly algorithm}

\begin{abstract}
We present an improved orderly algorithm for constructing all unlabelled lattices up to a given size, that is, an algorithm that constructs the minimal element of each isomorphism class relative to some total order.

Our algorithm employs a stabiliser chain approach for cutting branches of the search space that cannot contain a minimal lattice; to make this work, we grow lattices by adding a new layer at a time, as opposed to adding one new element at a time, and we use a total order that is compatible with this modified strategy.

The gain in speed is between one and two orders of magnitude.
As an application, we compute the number of unlabelled lattices on 20 elements.
\end{abstract}

\section{Introduction}\label{S:Introduction}

Enumerating all isomorphism classes of unlabelled lattices, in the sense of systematically constructing a complete list of isomorphism classes up to a certain size threshold, is a difficult combinatorial problem.
The number $u_n$ of isomorphism classes of unlabelled lattices on~$n$ elements grows faster than exponentially in~$n$~\cite{KlotzLucht,KleitmanWinston}, as does the number of (labelled) representatives of each isomorphism class.
Indeed, the largest value of~$n$ for which $u_n$ has been published previously is $n=19$~\cite{JipsenLawless,OEIS}.%

\medskip\noindent
When trying to enumerate combinatorial objects modulo isomorphism, one typically faces the problem that the number and the size of the isomorphism classes are so large that trying to weed out isomorphic objects through explicit isomorphism tests is out of the question.  Instead, an \emph{orderly algorithm} is needed, that is, an algorithm that traverses the search space in such a way that every isomorphism class is encountered exactly once.

A general strategy for the construction of isomorphism classes of combinatorial objects using \emph{canonical construction paths} was described in~\cite{McKay}.
Orderly algorithms for enumerating isomorphism classes of unlabelled lattices, as well as special subclasses of unlabelled lattices, were given in~\cite{HeitzigReinhold,JipsenLawless}.
The fastest published method currently is the one described in~\cite{JipsenLawless}; the computations of $u_{18}$ and $u_{19}$ reported in~\cite{JipsenLawless} took 26~hours respectively 19 days on 64 CPUs.

\medskip\noindent
The algorithms described in~\cite{HeitzigReinhold,JipsenLawless} follow a similar strategy:
\begin{inparaenum}[(i)]
 \item
   A total order~$<_\wt$ on all labelled lattices of a given size is defined.
 \item
   Starting from the (unique) $<_\wt$-minimal lattice on 2 elements, the $<_\wt$-minimal labelled representative of each isomorphism class of unlabelled lattices with at most~$n$ elements is constructed using a depth first search, where the children of a parent lattice are obtained by adding a single new element covering the minimal element of the parent lattice.
 \item
   For each parent lattice, one child is obtained for every choice for the covering set of the added element that yields a labelled lattice that is $<_\wt$-minimal in its isomorphism class of unlabelled lattices.
\end{inparaenum}

It is the test for $<_\wt$-minimality in step (iii) that takes most of the time:  While there are some necessary conditions that are easy to verify, ensuring that the newly constructed labelled lattice is indeed $<_\wt$-minimal in its isomorphism class requires checking the candidate covering set of the added element against all possible relabellings of the elements of the existing lattice; details are given in \autoref{S:Background}.
Basically, one has a certain permutation group that acts on a configuration space of covering sets, and one must verify that a given candidate is minimal in its orbit.

\medskip\noindent
It turns out that the elements of a $<_\wt$-minimal labelled lattice are arranged by \emph{levels} (cf.\ \autoref{S:Background}), and thus it is tempting to construct and test candidate covering sets of a new element level by level, exploiting the levellised structure for a divide-and-conquer approach; such an approach promises two advantages:
\begin{inparaenum}[(i)]
 \item
   The orbit of the restriction of a candidate covering set to a given level is potentially much smaller than the orbit of the complete covering set.
 \item
   The entire branch of the search space that corresponds to the candidate configuration of covers on the given level can potentially be discarded in a single test.
\end{inparaenum}

However, we shall see that  the constructions from~\cite{HeitzigReinhold,JipsenLawless} do not adapt well to this levellised approach:
In order to make the levellised approach work, we need to modify the depth-first-search to add one \emph{level} at a time as opposed to one \emph{element} at a time, and we need to modify the total order to be level-major.%

\bigskip\noindent
The structure of the paper is as follows:
In \autoref{S:Background}, we recall some results from~\cite{HeitzigReinhold,JipsenLawless} that are needed later, and we interpret the total order used in~\cite{HeitzigReinhold,JipsenLawless} as row-major.
In \autoref{S:ImprovedAlgorithm}, we describe our new construction using a level-major order and prove the results required to establish its correctness.
In \autoref{S:ImplementationResults}, we remark on implementation details and compare the performance of our new approach to that of those published in~\cite{HeitzigReinhold,JipsenLawless}.

\bigskip\noindent
We thank the Institute for Mathematics at the University of Seville (IMUS) for providing access to a 64-node 512 GB RAM computer that was used for an earlier version of this paper,
Peter Jipsen and Nathan Lawless for providing the code for the algorithm used in \cite{JipsenLawless} for timing comparisons,
Jukka Kohonen for suggesting to include a discussion of graded lattices,
and the anonymous referee for their valuable comments and suggestions.

\section{Background}\label{S:Background}

We start by giving a brief summary of the algorithms from~\cite{HeitzigReinhold,JipsenLawless}.
We refer to these sources for details.

\begin{definition}\label{D:lattice}
A finite bounded poset~$L$ is an \emph{$n$-poset}, if the elements of~$L$ are labelled $0,1,\ldots,n-1$, where~$0$ is a lower bound of~$L$ and~$1$ is an upper bound of~$L$.  An $n$-poset that is a lattice is called an \emph{$n$-lattice}.
To avoid confusion with the numerical order of integers, we denote the partial order of an $n$-poset~$L$ by $\isbelow[L]$ and $\isabove[L]$, or simply $\isbelow$ and $\isabove$ if the poset is obvious.
\end{definition}

\begin{notation}\label{N:lattice}
Assume that~$L$ is an $n$-poset.

For $a\in L$, we define the \emph{shadow} of~$a$ as $\downset{a} = \downset[L]{a} = \{ x\in L : x\isbelow[L] a \}$ and the \emph{shade} of~$a$ as $\upset{a} = \upset[L]{a} = \{ x\in L : x\isabove[L] a \}$.
For $A\subseteq L$, we define $\upset{A} = \upset[L]{A} = \bigcup_{a\in A}\upset[L]{a}$ as well as $\downset{A} = \downset[L]{A} = \bigcup_{a\in A}\downset[L]{a}$.

We say that~$a\in L$ has \emph{depth} $\dep(a) = \dep[L](a) = p$, if the maximum length of any chain from~$1$ to~$a$ in~$L$ is~$p+1$.
Given a non-negative integer~$k$, we call $\lev_k(L) = \{ a\in L : \dep[L](a) = k \}$ the \emph{$k$-th level} of~$L$.
We say that~$L$ is \emph{levellised}, if $\dep[L](i) \le \dep[L](j)$ holds for all $0<i\le j<n$.

For $a,b\in L$, we write $a\iscoveredby_L b$ (or simply $a\iscoveredby b$) if~$a$ is \emph{covered} by~$b$ in~$L$, that is, $a\isbelow_L b$ holds and $a\isbelow_L x\isbelow_L b$ implies $x=a$ or $x=b$.
We denote the \emph{covering set} of~$a\in L$ by $\cov a = \cov[L] a = \{ x\in L : a \iscoveredby x \}$, and we say that $a\in L$ is an \emph{atom} in~$L$ if $0\iscoveredby[L] a$ holds.

If~$L$ is a lattice and $a,b\in L$, we denote the least common upper bound of~$a$ and~$b$ in~$L$ by $a\join_L b$ (or simply $a\join b$), and the greatest common lower bound of~$a$ and~$b$ in~$L$ by $a\meet_L b$ (or simply $a\meet b$).

\end{notation}

\subsection{Canonical representatives}\label{SS:CanonicalRepresentatives}

The idea of an orderly algorithm is to construct all those lattices that are minimal, with respect to a suitable total order, in their isomorphism class.  In this section, we recall the total order used in~\cite{HeitzigReinhold,JipsenLawless} and some of its properties.

\begin{definition}\label{D:HR-weight}
Let~$L$ be an $n$-poset.
\begin{enumerate}[(a)]
 \item
   For $A\subseteq L$, we define $\wt_L(A) = \sum\limits_{j\in A} 2^j$.
 \item
   For $i\in L$, we define $\wt_L(i) = \wt_L\big(\cov[L] i\big)$.
 \item
   We define $\wt(L) = \big( \wt_L(2),\wt_L(3),\ldots,\wt_L(n-1)\big)$.
\end{enumerate}
Ordering $n$-lattices lexicographically with respect to $\wt(L)$, we obtain a total order $<_\wt$ on the set of all $n$-lattices.
We call an $n$-lattice \emph{$<_\wt$-minimal}, if it is minimal with respect to~$<_\wt$ in its isomorphism class.
\end{definition}

\begin{remark}\label{R:HR-order}
An $n$-poset~$L$ is completely defined by its covering relation.
Indeed, the upper bound~$1$ is not covered by any element, and it covers precisely those elements that are not covered by any other element.  Similarly, the lower bound~$0$ covers no element, and it is covered precisely by those elements that do not cover any other element.
Thus, $L$ is completely described by specifying the pairs $(i,j)$, for $1<i,j<n$, for which $i\iscoveredby j$ holds.

The latter information can be interpreted as an~$(n-2)$ by~$(n-2)$ matrix over~$\FF_2$,
where rows and columns are numbered $2,\ldots,n-1$, and the entry in row~$i$ and column~$j$ indicates whether or not $i\iscoveredby j$ holds.
The total order $<_\wt$ from \autoref{D:HR-weight} then amounts to a right-to-left row-major lexicographic order on the associated matrices; cf.\ \autoref{F:HR-order}.
 \begin{figure}[b]
    \begin{center}
      \begin{tikzpicture}[scale=0.30]
        \draw[step=1.0,black,thin] (0,0) grid (9,9);
        \node at (-0.5,8.5) {$\scriptstyle 2$};
        \node at (-1.0,0.5) {$\scriptstyle n-1$};
        \node at (0.5,-0.5) {$\scriptstyle 2$};
        \node at (10.5,-0.5) {$\scriptstyle n-1$};
        \draw (8.5,-0.2) -- (8.5,-0.5) -- (9.5,-0.5);
        \draw[-<,cyan] (0.5,0.5) -- (4.5,0.5);
        \draw[cyan] (4.5,0.5) -- (8.5,0.5);
        \draw[cyan,dashed] (8.5,0.5) -- (0.5,1.5);
        \draw[-<,cyan] (0.5,1.5) -- (4.5,1.5);
        \draw[cyan] (4.5,1.5) -- (8.5,1.5);
        \draw[cyan,dashed] (8.5,1.5) -- (0.5,2.5);
        \draw[-<,cyan] (0.5,2.5) -- (4.5,2.5);
        \draw[cyan] (4.5,2.5) -- (8.5,2.5);
        \draw[cyan,dashed] (8.5,2.5) -- (0.5,3.5);
        \draw[-<,cyan] (0.5,3.5) -- (4.5,3.5);
        \draw[cyan] (4.5,3.5) -- (8.5,3.5);
        \draw[cyan,dashed] (8.5,3.5) -- (0.5,4.5);
        \draw[-<,cyan] (0.5,4.5) -- (4.5,4.5);
        \draw[cyan] (4.5,4.5) -- (8.5,4.5);
        \draw[cyan,dashed] (8.5,4.5) -- (0.5,5.5);
        \draw[-<,cyan] (0.5,5.5) -- (4.5,5.5);
        \draw[cyan] (4.5,5.5) -- (8.5,5.5);
        \draw[cyan,dashed] (8.5,5.5) -- (0.5,6.5);
        \draw[-<,cyan] (0.5,6.5) -- (4.5,6.5);
        \draw[cyan] (4.5,6.5) -- (8.5,6.5);
        \draw[cyan,dashed] (8.5,6.5) -- (0.5,7.5);
        \draw[-<,cyan] (0.5,7.5) -- (4.5,7.5);
        \draw[cyan] (4.5,7.5) -- (8.5,7.5);
        \draw[cyan,dashed] (8.5,7.5) -- (0.5,8.5);
        \draw[-<,cyan] (0.5,8.5) -- (4.5,8.5);
        \draw[cyan] (4.5,8.5) -- (8.5,8.5);
      \end{tikzpicture}
    \end{center}\vspace{-4.5ex}
    \caption{Interpreting the order $<_\wt$ as right-to-left row-major lexicographic order on the matrices specifying the covering relation.
    Note that adding a row and a column for~$1$ does not affect the order:
    All entries in the added row are~$0$, and each entry in the added column is determined by the other entries in the same row, and it is checked after those in the lexicographic comparison.}\label{F:HR-order}
 \end{figure}
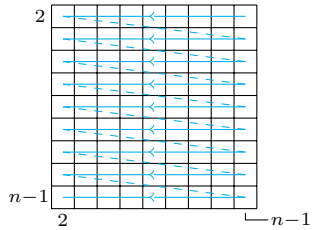
\end{remark}

\begin{theorem}[{\cite[Theorem 1]{HeitzigReinhold}}]\label{T:canonical-levellised}
If~$L$ is a $<_\wt$-minimal $n$-lattice and one has $0<i\le j<n$, then $\dep[L](i) \le \dep[L](j)$ holds.
\end{theorem}

\begin{corollary}\label{C:canonical-levellised}
If~$L$ is a $<_\wt$-minimal $n$-lattice and one has $0<i$ and $i\iscoveredby j$, then $j<i$ holds.
\end{corollary}

\begin{remark}\label{R:canonical-levellised}
\autoref{T:canonical-levellised} and \autoref{C:canonical-levellised} say that a $<_\wt$-minimal $n$-lattice~$L$ is levellised, that is, that the non-minimal levels of~$L$ are filled by elements labelled in their numerical order; cf.\ \autoref{F:levellised}.
 \begin{figure}[b]
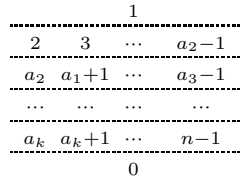

   \[\begin{array}{c@{\hspace{0.5em}}c@{\hspace{0.5em}}c@{\hspace{1.25em}}c}
                        &                     & \scriptstyle 1     &                    \\[-0.1ex] \hdashline[1pt/1pt]
     \scriptstyle 2 & \scriptstyle 3      & \scriptstyle\cdots & \scriptstyle a_2-1 \\[-0.1ex] \hdashline[1pt/1pt]
     \scriptstyle a_2   & \scriptstyle a_1+1  & \scriptstyle\cdots & \scriptstyle a_3-1 \\[-0.1ex] \hdashline[1pt/1pt]
     \scriptstyle\cdots & \scriptstyle\cdots  & \scriptstyle\cdots & \scriptstyle\cdots \\[-0.1ex] \hdashline[1pt/1pt]
     \scriptstyle a_k   & \scriptstyle a_k+1  & \scriptstyle\cdots & \scriptstyle n-1   \\[-0.1ex] \hdashline[1pt/1pt]
                        &                     & \scriptstyle 0
  \end{array}\]
  \vspace{-4.5ex}
  \caption{A levellised $n$-lattice; the dashed lines separate the levels.}\label{F:levellised}
 \end{figure}
\end{remark}

\begin{remark}\label{R:relabelling-levellised}
Any relabelling $\pi\in\Sym(\{0,\ldots,n-1\})$ of the elements of an $n$-lattice~$L$ preserves the levels, in the sense that $\dep[\pi(L)](\pi(i))=\dep[L](i)$ holds for every $i\in L$.  In particular, if~$L$ is a levellised lattice with the elements labelled as in \autoref{F:levellised}, then the relabelled lattice $\pi(L)$ is levellised if and only if
\begin{align*}
   \pi \in & \,\,\Sym(\{0\})
             \times \Sym(\{1\})
             \times \Sym(\{2,\ldots,a_2-1\}) \times  \\
           & \times \Sym(\{a_2,\ldots,a_3-1\})
             \times \cdots
             \times \Sym(\{a_k,\ldots,n-1\})
\end{align*}
holds.
\end{remark}

\begin{theorem}[{\cite[Theorem 2]{HeitzigReinhold}}]\label{T:canonical-weights}
If~$L$ is a $<_\wt$-minimal $n$-lattice, one has
\[
   \wt_L(2) \le \wt_L(3) \le \ldots \le \wt_L(n-1)
   \;.
\]
\end{theorem}

\begin{remark}\label{R:canonical-weight}
\autoref{T:canonical-weights} says that the rows of the matrices describing the covering relation of a $<_\wt$-minimal $n$-lattice~$L$ (cf.\ \autoref{F:HR-order}) are sorted in non-decreasing order with respect to a (right-to-left) lexicographic order on the~rows.
\end{remark}

\subsection{Incremental construction}\label{SS:IncrementalConstruction}

The algorithms from~\cite{HeitzigReinhold,JipsenLawless} work by traversing a tree of $<_\wt$-minimal $n$-lattices in a depth-first manner; the root of the tree is the unique 2-lattice, and an $(n+1)$-lattice~$\widetilde{L}$ is a descendant of the $<_\wt$-minimal $n$-lattice~$L$, if $\widetilde{L}$ is obtained from~$L$ by adding a new cover of~$0$ (labelled~$n$) and~$\widetilde{L}$ is $<_\wt$-minimal.

The lattice~$\widetilde{L}$ is determined by~$L$ and the covering set of the new element~$n$; the possible choices for the latter can be characterised effectively.

\begin{definition}[\cite{HeitzigReinhold}]\label{D:lattice-antichain}
If~$L$ is an $n$-lattice, a non-empty antichain $A\subseteq L\setminus\{0\}$ is called a \emph{lattice-antichain} for~$L$, if $a\meet[L] b \in \{0\} \cup (\upset[L]{A})$ holds for any $a,b\in\upset[L]{A}$.
\end{definition}

\begin{remark}\label{R:lattice-antichain}
 To test the condition in \autoref{D:lattice-antichain}, it is clearly sufficient to verify that
 $a\meet[L] b \in \upset[L]{A}$ holds for those pairs~$(a,b)$ that are minimal in the set
 \[
    \big\{ (a,b) \in (\upset[L]{A}) \times (\upset[L]{A}) : a\meet[L] b \ne 0 \big\}
 \]
 with respect to the product partial order in~$L\times L$.
\end{remark}

\begin{theorem}[{\cite[Lemma 2]{HeitzigReinhold}}]\label{T:HR-incremental}
Let~$L$ be an $n$-lattice.  A subset $A\subseteq L\setminus\{0\}$ is a lattice-antichain for~$L$, if and only if~$L$ is a subposet of an $(n+1)$-lattice~$L_A$ in which $0\iscoveredby[L_A] n$ (that is,~$n$ is an atom in~$L_A$) and $\cov[L_A] n = A$ hold.
\end{theorem}

\begin{remark}\label{R:incremental-unique}\label{R:incremental-weight-depth}
In the situation of \autoref{T:HR-incremental}, it is clear that the pair $(L,A)$ uniquely determines~$L_A$ and vice versa.
Moreover, the covering relation of~$L_A$ is obtained from the covering relation of~$L$ by
\begin{enumerate}[(i)]\addtolength{\itemsep}{-1ex}\vspace{-1ex}
 \item
   adding the pair $(0,n)$;
 \item
   adding all pairs of the form $(n,a)$ for $a\in A$; and
 \item
   removing all pairs of the form $(0,a)$ for $a\in A$ that are present.
\end{enumerate}
As mentioned in \autoref{R:HR-order}, the covers of~$0$ need not be stored explicitly; in this case, only step (ii) is needed.

Indeed, this definition of~$L_A$ makes sense for any $n$-poset~$L$ and any~$A\subseteq L$.
It is obvious from the definitions that one has $\wt_{L_A}(n) = \wt_L(A)$.  Moreover, for $1\le i < n$, we have $i\not\isbelow[L_A] n$ and thus $\wt_{L_A}(i) = \wt_L(i)$ and $\dep[L_A](i) = \dep[L](i)$.
\end{remark}

\medskip\noindent
The following two results are consequences of \autoref{R:incremental-weight-depth} and \autoref{T:canonical-weights}.

\begin{corollary}[{\cite[\S3]{HeitzigReinhold}}]\label{C:HR-incremental}
If~$L$ is an $n$-lattice and~$A$ is a lattice-antichain for~$L$ such that~$L_A$ is $<_\wt$-minimal, then~$L$ is $<_\wt$-minimal.
\end{corollary}

\begin{corollary}[{\cite[\S5]{HeitzigReinhold}}]\label{C:HR-incremental-antichain}
If~$L$ is an $n$-lattice and~$A$ is a lattice-antichain for~$L$ such that~$L_A$ is $<_\wt$-minimal, then~one has $A\cap\big(\lev_{k-1}(L)\cup\lev_k(L)\big)\neq\emptyset$ for $k=\dep[L](n-1)$.
\end{corollary}

\subsection{Testing for canonicity}\label{SS:Canonical}

In the light of \autoref{C:HR-incremental-antichain}, there are two cases to consider for testing whether the descendant $L_A$ of a $<_\wt$-minimal $n$-lattice~$L$ defined by a lattice-antichain~$A$ for~$L$ with $\dep[L](n-1)=k$ is $<_\wt$-minimal:
\begin{enumerate}[(A)]
 \item $A\cap\lev_k(L)\neq\emptyset$, that is, $\dep[L_A](n) = k+1$
 \item $A\cap\lev_k(L)=\emptyset\neq A\cap\lev_{k-1}(L)$, that is, $\dep[L_A](n) = k$
\end{enumerate}

\medskip\noindent
\textbf{Case (A):}

In this case, the new element~$n$ forms a separate level of~$L_A$; cf.\ \autoref{F:HR-large}.
Since~$L_A$ is levellised by construction and the non-minimal elements of~$L$ correspond to the levels $0,\ldots,k$ of~$L_A$, any relabelling~$\pi$ of~$L_A$ for which~$\pi(L_A)$ is levellised must fix~$n$ and induce a relabelling of~$L$ by \autoref{R:relabelling-levellised}.  By \autoref{R:incremental-weight-depth} and the definition of the lexicographic order, $\wt(\pi(L_A)) < \wt(L_A)$ implies $\wt(\pi(L)) \le \wt(L)$.
By \autoref{C:HR-incremental}, the latter implies $\wt(\pi(L)) = \wt(L)$ and thus $\pi(L)=L$.

To test whether~$L_A$ is $<_\wt$-minimal, it is thus sufficient to check that
\begin{equation}\label{EQ:largeAntichainMinimal}
 \wt_L(A) = \min \big\{\wt_L(\pi(A)) : \pi\in\Stab(L)\times\Sym(\{n\}) \cong \Stab(L) \big\}
\end{equation}
holds, again using \autoref{R:incremental-weight-depth}.
The latter condition can, for instance, be verified by computing the orbit $A^{\Stab(L)}$ of~$A$ under the action of $\Stab(L)$ as the closure of the set $\{A\}$ under the action of a generating set for $\Stab(L)$; this is the approach taken in~\cite{HeitzigReinhold}.
Alternatively, one can compute a \emph{canonical labelling} of the lattice~$L_A$ and use it to test whether~$L_A$ lies on a canonical construction path; this is the approach taken in~\cite{JipsenLawless}.

\bigskip\noindent
\textbf{Case (B):}

In this case, the new element~$n$ is added to the lowest existing non-trivial level of~$L$; cf.\ \autoref{F:HR-small}.  Thus, a relabelling~$\pi$ of~$L_A$ for which~$\pi(L_A)$ is levellised \emph{need not} fix~$n$ and\commS{ST: nor} induce a relabelling of~$L$.
It will, however, induce a relabelling of the lattice~$L'$ induced by the levels $0,\ldots,k-1,k+1$ of~$L$ (or~$L_A$), and one has $\pi(L') = L'$ by the same arguments as in the previous case.

If the lowest non-trivial level of~$L_A$ contains the elements $a_k,\ldots,n$, checking whether~$L_A$ is $<_\wt$-minimal means testing that $\big(\wt_{L_A}(a_k),\ldots,\wt_{L_A}(n)\big)$ is lexicographically minimal in its orbit under the group $\Stab(L')\times\Sym(\{a_k,\ldots,n\})$;
the latter can again be done by an explicit computation of the orbit, or by using a canonical labelling.

Observe that one has $\cov[L_A] i \subseteq L'$ for $a_k\le i\le n$, and that the action of $\pi\in\Stab(L')\times\Sym(\{a_k,\ldots,n\})$ not only modifies the individual weights $\wt_{L_A}(i)$ (by relabelling the elements of~$L'$), but also permutes their positions in the sequence (by acting on $\{a_k,\ldots,n\}$).

\begin{figure}[b]
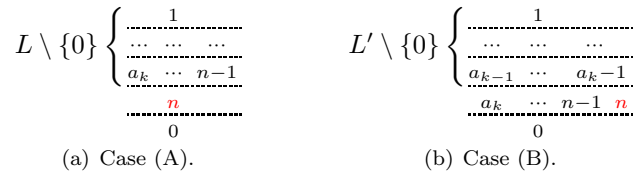
%
\centering
\subfigure[Case~(A).]{\label{F:HR-large}$\begin{array}[t]{@{}c@{}c@{\hspace{0.5em}}c@{\hspace{0.5em}}c@{\,}}
             \multirow{3}{*}{$L\setminus\{0\}\left\{\rule{0pt}{17pt}\right.$} &  & \scriptstyle 1 &
                \\[-0.25ex] \cdashline{2-4}[1pt/1pt]
             & \scriptstyle\cdots & \scriptstyle\cdots & \scriptstyle\cdots
                \rule[6pt]{0pt}{0pt}\\[-0.25ex] \cdashline{2-4}[1pt/1pt]
             & \scriptstyle a_k & \scriptstyle\cdots & \scriptstyle n-1
                \rule[6pt]{0pt}{0pt}\\[-0.25ex] \cdashline{2-4}[1pt/1pt]
             &  & \scriptstyle \textcolor{red}{n}
                \rule[6pt]{0pt}{0pt}\\[-0.25ex] \cdashline{2-4}[1pt/1pt]
             & & \scriptstyle 0
          \end{array}$}\qquad\qquad
\subfigure[Case (B).]{\label{F:HR-small}$\begin{array}[t]{@{}c@{}c@{\hspace{0.5em}}c@{\hspace{0.5em}}c@{\,}}
             \multirow{3}{*}{$L'\setminus\{0\}\left\{\rule{0pt}{17pt}\right.$} &  & \scriptstyle 1 &
                \\[-0.25ex] \cdashline{2-4}[1pt/1pt]
             & \scriptstyle\cdots & \scriptstyle\cdots & \scriptstyle\cdots
                \rule[6pt]{0pt}{0pt}\\[-0.25ex] \cdashline{2-4}[1pt/1pt]
             & \scriptstyle a_{k-1} & \scriptstyle\cdots & \hspace{0.5em}\scriptstyle a_k-1
                \rule[6pt]{0pt}{0pt}\\[-0.25ex] \cdashline{2-4}[1pt/1pt]
             & \scriptstyle a_k & \scriptstyle\cdots & \scriptstyle n-1\hspace{0.5em}\textcolor{red}{\scriptstyle n}
                \rule[6pt]{0pt}{0pt}\\[-0.25ex] \cdashline{2-4}[1pt/1pt]
             & & \scriptstyle 0
          \end{array}$}
\caption{The lattice $L_A$ obtained by adding a new cover~$n$ of~$0$ to the $n$-lattice~$L$.}
\label{F:HR}
\end{figure}

\subsection{Vertically indecomposable lattices}\label{SS:Indecomposable}

\begin{definition}\label{D:indecomposable}
 An $n$-lattice~$L$ is \emph{vertically decomposable}, if there exists an element $i\in L\setminus\{0,1\}$ that is comparable to every other element of~$L$.  Otherwise,~$L$ is \emph{vertically indecomposable}.
\end{definition}

\noindent
One can speed up the construction by restricting to lattices that are vertically indecomposable; a straightforward recursion makes it possible to recover all lattices from the vertically indecomposable ones.
We refer to \cite[\S5]{HeitzigReinhold} for details.

\section{An improved algorithm}\label{S:ImprovedAlgorithm}

The test for minimality of $\wt_L(A)$ or $\big(\wt_{L_A}(a_k),\ldots,\wt_{L_A}(n)\big)$ in their orbit under the acting permutation group is the most time consuming part of the construction and thus an obvious target for improvement.

In \autoref{SS:StabiliserChains}, we sketch the basic idea for a more efficient algorithm, but we will see that the construction of \cite{HeitzigReinhold,JipsenLawless} has to be modified to make this idea work.
We describe our modified construction in \autoref{SS:Levellised}.

\subsection{Stabiliser chain approach}\label{SS:StabiliserChains}

Case (A) from \autoref{SS:Canonical} suggests a possible approach, namely the use of a standard technique from computational group theory: stabiliser chains.

Since $\Stab(L)$ preserves each level of~$L$ by \autoref{R:relabelling-levellised} and~$L$ is levellised, it is tempting to construct and test lattice-antichains level by level:
Defining $S_k := \Stab(L)$ as well as $A_d := A\cap \lev_d(L)$ and
$S_{d-1} := S_d \cap \Stab(A_d)$ for $d = k,\ldots,1$, condition~\eqref{EQ:largeAntichainMinimal} is equivalent to the following condition:
\begin{equation}\label{EQ:largeAntichainMinimal-chain}
 \wt_L(A_d) = \min \big\{\wt_L(\pi(A_d)) : \pi\in S_d \big\} \mbox{\quad for $d = k,\ldots,1$}
\end{equation}
The sets $A_d$ for $d = k,\ldots,1$ can be constructed and tested one at a time, which offers two advantages:
Firstly, if the test at level~$d$ fails, the sets for the levels $d-1,\ldots,1$ don't have to be constructed; an entire branch of the search space is discarded in one step.
Secondly, even if the test succeeds on all levels, the cost of testing condition~\eqref{EQ:largeAntichainMinimal-chain} by computing the orbits $A_d^{S_d}$ for $d=k,\ldots,1$ is proportional to $\sum_{d=1}^k |A_d^{S_d}|$, and thus in general much smaller than  $|A^{\Stab(L)}| = \prod_{d=1}^k |A_d^{S_d}|$, which is the cost of testing condition~\eqref{EQ:largeAntichainMinimal} directly by computing the orbit $A^{\Stab(L)}$.

\bigskip\noindent
However, when trying to use a similar approach for Case (B), we run into problems:
We must compare $$\big(\wt_{L_A}(a_k),\ldots,\wt_{L_A}(n)\big) = \big(\wt_L(a_k),\ldots,\wt_L(n-1),\wt_L(A)\big)$$ lexicographically to its images under the elements of the acting permutation group $\Stab(L')\times\Sym(\{a_k,\ldots,n\})$, but if~$A$ has only been constructed partially, $\wt_{L_A}(n)=\wt_L(A)$ is not completely determined; in the interpretation of \autoref{R:HR-order} and \autoref{F:HR-order}, the leftmost entries of the last row of the binary matrix corresponding to~$L_A$ are undefined.

The elements of the group $\Stab(L')\times\Sym(\{a_k,\ldots,n\})$ can permute the rows of this matrix, so the position of the undefined entries will vary.  Clearly, the lexicographic comparison of the two matrices must stop once it reaches an entry that is undefined in one of the matrices being compared; in this situation, the order of the two matrices cannot be decided on the current level.

The problem is that the position in the matrix at which the lexicographic comparison must stop depends on the relabelling that is applied (cf.\ \autoref{F:HR-layered-small}).
A consequence of this is that the subset of elements of $\Stab(L')\times\Sym(\{a_k,\ldots,n\})$ for which the parts of the matrices that can be compared are equal does not form a subgroup, so applying a stabiliser chain approach is not possible.

\begin{figure}[t]
  \begin{center}
   \begin{tikzpicture}[scale=0.30]
     \draw[fill=lightgray] (0,0) rectangle (3,1);
     \draw[step=1,black,very thin] (0,0) grid (9,9);
     \draw[xstep=3,ystep=9,black,very thick] (0,0) grid (9,9);
     \node at (-0.75,8.5) {$\scriptstyle a_k$};
     \node at (-0.75,0.5) {$\scriptstyle n$};
     \node at (0.5,-0.5) {$\scriptstyle 2$};
     \node at (10.5,-0.5) {$\scriptstyle a_k-1$};
     \draw (8.5,-0.2) -- (8.5,-0.5) -- (9.25,-0.5);
     \draw[-<,cyan] (3.5,0.5) -- (4.5,0.5);
     \draw[cyan] (4.0,0.5) -- (8.5,0.5);
     \draw[cyan,dashed] (8.5,0.5) -- (0.5,1.5);
     \draw[-<,cyan] (0.5,1.5) -- (4.5,1.5);
     \draw[cyan] (4.5,1.5) -- (8.5,1.5);
     \draw[cyan,dashed] (8.5,1.5) -- (0.5,2.5);
     \draw[-<,cyan] (0.5,2.5) -- (4.5,2.5);
     \draw[cyan] (4.5,2.5) -- (8.5,2.5);
     \draw[cyan,dashed] (8.5,2.5) -- (0.5,3.5);
     \draw[-<,cyan] (0.5,3.5) -- (4.5,3.5);
     \draw[cyan] (4.5,3.5) -- (8.5,3.5);
     \draw[cyan,dashed] (8.5,3.5) -- (0.5,4.5);
     \draw[-<,cyan] (0.5,4.5) -- (4.5,4.5);
     \draw[cyan] (4.5,4.5) -- (8.5,4.5);
     \draw[cyan,dashed] (8.5,4.5) -- (0.5,5.5);
     \draw[-<,cyan] (0.5,5.5) -- (4.5,5.5);
     \draw[cyan] (4.5,5.5) -- (8.5,5.5);
     \draw[cyan,dashed] (8.5,5.5) -- (0.5,6.5);
     \draw[-<,cyan] (0.5,6.5) -- (4.5,6.5);
     \draw[cyan] (4.5,6.5) -- (8.5,6.5);
     \draw[cyan,dashed] (8.5,6.5) -- (0.5,7.5);
     \draw[-<,cyan] (0.5,7.5) -- (4.5,7.5);
     \draw[cyan] (4.5,7.5) -- (8.5,7.5);
     \draw[cyan,dashed] (8.5,7.5) -- (0.5,8.5);
     \draw[-<,cyan] (0.5,8.5) -- (4.5,8.5);
     \draw[cyan] (4.5,8.5) -- (8.5,8.5);
   \end{tikzpicture}
   \qquad\qquad
   \begin{tikzpicture}[scale=0.30]
     \draw [fill=lightgray] (0,4) rectangle (3,5);
     \draw[step=1,black,very thin] (0,0) grid (9,9);
     \draw[xstep=3,ystep=9,black,very thick] (0,0) grid (9,9);
     \node at (-0.75,8.5) {$\scriptstyle a_k$};
     \node at (-0.75,0.5) {$\scriptstyle n$};
     \node at (0.5,-0.5) {$\scriptstyle 2$};
     \node at (10.5,-0.5) {$\scriptstyle a_k-1$};
     \draw (8.5,-0.2) -- (8.5,-0.5) -- (9.25,-0.5);
     \draw[-<,cyan] (3.5,4.5) -- (4.5,4.5);
     \draw[cyan] (4.0,4.5) -- (8.5,4.5);
     \draw[cyan,dashed] (8.5,4.5) -- (0.5,5.5);
     \draw[-<,cyan] (0.5,5.5) -- (4.5,5.5);
     \draw[cyan] (4.5,5.5) -- (8.5,5.5);
     \draw[cyan,dashed] (8.5,5.5) -- (0.5,6.5);
     \draw[-<,cyan] (0.5,6.5) -- (4.5,6.5);
     \draw[cyan] (4.5,6.5) -- (8.5,6.5);
     \draw[cyan,dashed] (8.5,6.5) -- (0.5,7.5);
     \draw[-<,cyan] (0.5,7.5) -- (4.5,7.5);
     \draw[cyan] (4.5,7.5) -- (8.5,7.5);
     \draw[cyan,dashed] (8.5,7.5) -- (0.5,8.5);
     \draw[-<,cyan] (0.5,8.5) -- (4.5,8.5);
     \draw[cyan] (4.5,8.5) -- (8.5,8.5);
   \end{tikzpicture}
  \end{center}\vspace{-4.5ex}
  \caption{Lexicographic comparison of a lattice~$L_A$ obtained from a partially constructed lattice-antichain~$A$ (left) and a relabelling (right); only the relevant parts of the matrices are shown.  Thick lines indicate the boundaries between levels.  The parts of the lattice-antichain not yet constructed are shown in grey.}\label{F:HR-layered-small}
\end{figure}
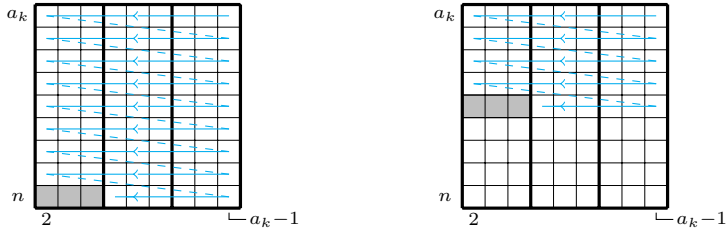

\subsection{Levellised construction}\label{SS:Levellised}

The analysis at the end of the preceding section indicates that the problem is that possible relabellings can swap an element whose covering set is only partially determined with an element whose covering set is completely determined, or in other words, that we add a new element to an existing level of~$L$.

The idea for solving this problem is simple:
Rather than adding one element at a time, possibly to an already existing level, we only ever add an entire level at a time; that way, the problem of adding elements to an existing level is avoided.

\medskip\noindent
To make the stabiliser chain approach work in this setting, we must use a total order that compares parts of the covering sets in the same order in which they are constructed; that is, we have to compare the entries of the matrices describing the covering relations in level-major order.

\begin{notation}\label{N:parentLattice}
 For a levellised $n$-lattice~$L$ with $n>2$ and $\dep[L](n-1)=k$, let~$L'$ denote the lattice induced by the levels $0,\ldots,k-1,k+1$ of~$L$, that is, the lattice obtained from~$L$ by removing its last non-trivial level.
 Note that~$L'$ is a levellised $n'$-lattice for some $n'<n$.
\end{notation}

\medskip\noindent
The total order we are about to define uses the partition of the covering set of each element according to levels.

\begin{definition}\label{D:levellisedCoveringSet}
 Given a levellised $n$-lattice~$L$ with $\dep[L](n-1)=k$, an element $i\in L\setminus\{0\}$ and an integer $d\in\{1,\ldots,k-1\}$, we define $\cov[L][d] i = \big( \cov[L] i \big) \cap \lev_d(L)$.
\end{definition}

\begin{definition}\label{D:lattice-order}
 Using induction on~$n$, we define a relation~$<$ on the set of levellised $n$-lattices that are isomorphic as unlabelled lattices as follows:

 Given $n>2$ and levellised $n$-lattices~$L_1$ and~$L_2$ that are isomorphic as unlabelled lattices, we say $L_1 < L_2$, if one of the following holds:

 \begin{itemize}
  \item
    $L_1' < L_2'$
  \item
    $L_1' = L_2' = L'$ and, denoting $\dep[L_1](n-1)=\dep[L_2](n-1)=k$ and $\lev_k(L_1) = \lev_k(L_2) = \{a_k,\ldots,n-1\}$, there exist $\ell\in\{1,\ldots,k-1\}$ as well as $i\in \{a_k,\ldots,n-1\}$ such that both of the following hold:
    \begin{itemize}[\quad$\circ$]\vspace{-\topsep}
     \item
       $\wt_{L'}\big(\cov[L_1][d] j\big) = \wt_{L'}\big(\cov[L_2][d] j\big)$ if $d>\ell$, or $d=\ell$ and $j\in\{a_k,\ldots,i-1\}$
     \item
       $\wt_{L'}\big(\cov[L_1][\ell] i\big) < \wt_{L'}\big(\cov[L_2][\ell] i\big)$
    \end{itemize}
 \end{itemize}
\end{definition}

\begin{remark}\label{R:lattice-order}
The relation from \autoref{D:lattice-order} corresponds to a level-major lexicographic comparison of the binary matrices describing the covering relations of~$L_1$ and~$L_2$ as illustrated in \autoref{F:lattice-order}.

As~$L_1$ and~$L_2$ are levellised, both matrices are lower block triagonal, with the blocks defined by the levels.
Moreover, as~$L_1$ and~$L_2$ are isomorphic as unlabelled lattices, the block structures of both matrices are identical.

Notice also that the matrix describing the covering relation of~$L_1'$ (respectively~$L_2'$) is obtained from that of~$L_1$ (respectively~$L_2$) by removing the lowest row and the rightmost column of blocks.
 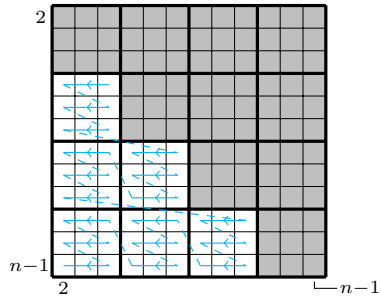
\begin{figure}[b]
    \begin{center}
      \begin{tikzpicture}[scale=0.30]
        \draw [fill=lightgray] (0,9) rectangle (12,12);
        \draw [fill=lightgray] (3,6) rectangle (12,12);
        \draw [fill=lightgray] (6,3) rectangle (12,12);
        \draw [fill=lightgray] (9,0) rectangle (12,12);
        \draw[step=1.0,black,very thin] (0,0) grid (12,12);
        \draw[step=3.0,black,very thick] (0,0) grid (12,12);
        \node at (-0.5,11.5) {$\scriptstyle 2$};
        \node at (-1.0,0.5) {$\scriptstyle n-1$};
        \node at (0.5,-0.5) {$\scriptstyle 2$};
        \node at (13.5,-0.5) {$\scriptstyle n-1$};
        \draw (11.5,-0.2) -- (11.5,-0.5) -- (12.5,-0.5);

        \draw[->,cyan] (2.5,8.5) -- (1.5,8.5);
        \draw[cyan] (1.5,8.5) -- (0.5,8.5);
        \draw[cyan,dashed] (0.5,8.5) -- (2.5,7.5);
        \draw[->,cyan] (2.5,7.5) -- (1.5,7.5);
        \draw[cyan] (1.5,7.5) -- (0.5,7.5);
        \draw[cyan,dashed] (0.5,7.5) -- (2.5,6.5);
        \draw[->,cyan] (2.5,6.5) -- (1.5,6.5);
        \draw[cyan] (1.5,6.5) -- (0.5,6.5);
        \draw[cyan,dashed] (0.5,6.5) -- (5.5,5.5);

        \draw[->,cyan] (5.5,5.5) -- (4.5,5.5);
        \draw[cyan] (4.5,5.5) -- (3.5,5.5);
        \draw[cyan,dashed] (3.5,5.5) -- (5.5,4.5);
        \draw[->,cyan] (5.5,4.5) -- (4.5,4.5);
        \draw[cyan] (4.5,4.5) -- (3.5,4.5);
        \draw[cyan,dashed] (3.5,4.5) -- (5.5,3.5);
        \draw[->,cyan] (5.5,3.5) -- (4.5,3.5);
        \draw[cyan] (4.5,3.5) -- (3.5,3.5);
        \draw[cyan,dashed] (3.5,3.5) -- (2.5,5.5);

        \draw[->,cyan] (2.5,5.5) -- (1.5,5.5);
        \draw[cyan] (2.5,5.5) -- (0.5,5.5);
        \draw[cyan,dashed] (0.5,5.5) -- (2.5,4.5);
        \draw[->,cyan] (2.5,4.5) -- (1.5,4.5);
        \draw[cyan] (1.5,4.5) -- (0.5,4.5);
        \draw[cyan,dashed] (0.5,4.5) -- (2.5,3.5);
        \draw[->,cyan] (2.5,3.5) -- (1.5,3.5);
        \draw[cyan] (1.5,3.5) -- (0.5,3.5);
        \draw[cyan,dashed] (0.5,3.5) -- (8.5,2.5);

        \draw[->,cyan] (8.5,2.5) -- (7.5,2.5);
        \draw[cyan] (7.5,2.5) -- (6.5,2.5);
        \draw[cyan,dashed] (6.5,2.5) -- (8.5,1.5);
        \draw[->,cyan] (8.5,1.5) -- (7.5,1.5);
        \draw[cyan] (7.5,1.5) -- (6.5,1.5);
        \draw[cyan,dashed] (6.5,1.5) -- (8.5,0.5);
        \draw[->,cyan] (8.5,0.5) -- (7.5,0.5);
        \draw[cyan] (7.5,0.5) -- (6.5,0.5);
        \draw[cyan,dashed] (6.5,0.5) -- (5.5,2.5);

        \draw[->,cyan] (5.5,2.5) -- (4.5,2.5);
        \draw[cyan] (4.5,2.5) -- (3.5,2.5);
        \draw[cyan,dashed] (3.5,2.5) -- (5.5,1.5);
        \draw[->,cyan] (5.5,1.5) -- (4.5,1.5);
        \draw[cyan] (4.5,1.5) -- (3.5,1.5);
        \draw[cyan,dashed] (3.5,1.5) -- (5.5,0.5);
        \draw[->,cyan] (5.5,0.5) -- (4.5,0.5);
        \draw[cyan] (4.5,0.5) -- (3.5,0.5);
        \draw[cyan,dashed] (3.5,0.5) -- (2.5,2.5);

        \draw[->,cyan] (2.5,2.5) -- (1.5,2.5);
        \draw[cyan] (2.5,2.5) -- (0.5,2.5);
        \draw[cyan,dashed] (0.5,2.5) -- (2.5,1.5);
        \draw[->,cyan] (2.5,1.5) -- (1.5,1.5);
        \draw[cyan] (1.5,1.5) -- (0.5,1.5);
        \draw[cyan,dashed] (0.5,1.5) -- (2.5,0.5);
        \draw[->,cyan] (2.5,0.5) -- (1.5,0.5);
        \draw[cyan] (1.5,0.5) -- (0.5,0.5);

      \end{tikzpicture}
    \end{center}\vspace{-4.5ex}
    \caption{Interpreting the order $<$ as level-major lexicographic order on the matrices specifying the covering relation.  Thick lines indicate the boundaries between levels.  The entries of the matrix shown in grey are zero.}\label{F:lattice-order}
 \end{figure}
\end{remark}

\begin{lemma}\label{L:lattice-order}
 The relation~$<$ from \autoref{D:lattice-order} is a total order on the set of levellised $n$-lattices that are isomorphic as unlabelled lattices.
\end{lemma}
\begin{proof}
 Verifying transitivity is routine.  Trichotomy holds, as for $s\in\{1,2\}$ and $i\in\{2,\ldots,n-1\}$ with
 $\dep[L_1](i) = \dep[L_2](i) = D$, one has
 \[
    \cov[L_s] i =
      \left\{ \begin{array}{ll}
        \bigcup\limits_{d=1}^{D-1} \cov[L_s][d] i
            & \text{;\quad if $\bigcup\limits_{d=1}^{D-1} \cov[L_s][d] i \ne \emptyset$} \\[4ex]
        \{1\} & \text{;\quad if $\bigcup\limits_{d=1}^{D-1} \cov[L_s][d] i = \emptyset$}
      \end{array}\right.
 \]
 whence $\cov[L_1] i = \cov[L_2] i$ holds in the situation $L_1' = L_2' = L'$ if and only if one has $\wt_{L'}\big(\cov[L_1][d] i\big) = \wt_{L'}\big(\cov[L_2][d] i\big)$ for $d=1,\ldots,D-1$.
\end{proof}

\begin{definition}\label{D:canonical}
An $n$-lattice~$L$ is \emph{canonical} if~$L$ is levellised and $<$-minimal among all levellised $n$-lattices that are isomorphic to~$L$ as unlabelled lattices.
\end{definition}

\begin{lemma}\label{L:canonical}
Every isomorphism class of unlabelled lattices on~$n$ elements contains a unique canonical $n$-lattice.
\end{lemma}
\begin{proof}
The set of representatives that are levellised $n$-lattices is clearly non-empty and finite, and~$<$ is a total order on this set by \autoref{L:lattice-order}.
\end{proof}

\medskip\noindent
The following \autoref{T:incremental-parent} is an analogue of \autoref{C:HR-incremental}.

\begin{lemma}\label{L:sublattice}
If $L$ is a levellised $(n+1)$-lattice for $n>1$, then $\overline{L} = L\setminus\{n\}$ is a levellised
$n$-lattice and $\dep[\overline{L}](i) = \dep[L](i)$ holds for $i=1,\ldots,n-1$.
Moreover, for $a,b\in \overline{L}$ one has the following identities:
 \[
    a \join[\overline{L}] b = a \join[L] b
    \quad\text{and}\quad
    a \meet[\overline{L}] b = \left\{ \begin{array}{ll}
                                         a \meet[L] b & \text{;\quad if $a \meet[L] b \ne n$} \\
                                         0            & \text{;\quad if $a \meet[L] b = n$}
                                      \end{array} \right.
 \]
\end{lemma}
\begin{proof}
 As~$L$ is levellised, $1\le i < n$ implies $\dep[L](i) \le \dep[L](n)$, so $i\not\isbelow[L]n$, and hence $\dep[\overline{L}](i) = \dep[L](i)$.  As~$L$ is levellised, so is~$\overline{L}$.

 A routine verification shows the identities for $a \join[\overline{L}] b$ and $a \meet[\overline{L}] b$; in particular, $\overline{L}$ is an $n$-lattice.
\end{proof}

\begin{theorem}\label{T:incremental-parent}
If~$L$ is a canonical $n$-lattice, then~$L'$ is canonical.
\end{theorem}
\begin{proof}
Iterated application of \autoref{L:sublattice} shows that~$L'$ is levellised.

Assume that~$\pi'$ is a relabelling of~$L'$ such that $\pi'(L')$ is levellised and one has $\pi'(L')<L'$.
We can trivially extend~$\pi'$ to a relabelling~$\pi$ of~$L$ such that $\big(\pi(L)\big)' = \pi'(L') < L'$, contradicting the assumption that~$L$ is canonical.
\end{proof}

\medskip\noindent
\autoref{T:incremental-parent} means that we can again construct a tree of canonical $n$-lattices in a depth-first manner; the root of the tree is the unique 2-lattice, and an $(n+m)$-lattice~$\widetilde{L}$ is a descendant of a canonical $n$-lattice~$L$, if $\widetilde{L}$ is obtained from~$L$ by adding a new level consisting of~$m$ new covers of~$0$ (labelled~$n,\ldots,n+m-1$) and~$\widetilde{L}$ is canonical.

The lattice~$\widetilde{L}$ is determined by~$L$ and the covering sets of the new elements~$n,\ldots,n+m-1$; the possible choices for the latter can again be characterised effectively using lattice-antichains, although this time, extra compatibility conditions are needed.  The following \autoref{T:incremental}, a generalisation of \autoref{T:HR-incremental}, makes this precise.

\begin{notation}\label{N:incremental}
Given $m\in\NN^+$, an $n$-poset~$L$, and $A_n,\ldots,A_{n+m-1}\subseteq L\setminus\{0\}$, let
$L_{A_n,\ldots,A_{n+m-1}} = \big(\cdots\big(L_{A_n}\big)\cdots\big)_{A_{n+m-1}}$ denote the $(n+m)$-poset obtained from~$L$ by adding~$m$ new atoms $n,\ldots,n+m-1$ with $\cov[L_{A_n,\ldots,A_{n+m-1}}] i = A_i$ for $i=n,\ldots,n+m-1$.
(See \autoref{R:incremental-weight-depth}.)
\end{notation}

\begin{lemma}\label{L:incremental}
Let~$L$ be a levellised $n$-poset with $\dep[L](n-1)=k$, let $m\in\NN^+$, and let $A_i\subseteq L\setminus\{0\}$ for $i=n,\ldots,n+m-1$.  The following are equivalent:
\begin{enumerate}[\upshape (A)]
 \item
   $\widetilde{L} = L_{A_n,\ldots,A_{n+m-1}}$ is a levellised $(n+m)$-poset with $\dep[\widetilde{L}](a)=\dep[L](a) \le k$ for $1\le a < n$ and $\dep[\widetilde{L}](n)=\ldots=\dep[\widetilde{L}](n+m-1)=k+1$.
 \item
   $A_i \cap \lev_k(L) \ne \emptyset$ holds for $n \le i < n+m$.
\end{enumerate}
\end{lemma}
\begin{proof}
 As~$L$ is levellised, one has $\dep[L](a) \le \dep[L](n-1) = k$ for all $a\in L\setminus\{0\}$, and induction using \autoref{R:incremental-weight-depth} shows that $\dep[\widetilde{L}](a) = \dep[L](a) \le k$ holds for all $a\in L\setminus\{0\}$, which implies $\dep[\widetilde{L}](i)\le k+1$ for $i=n,\ldots,n+m-1$ by construction.
 Thus, for any $i=n,\ldots,n+m-1$, one has $\dep[\widetilde{L}](i)=k+1$ if and only if $\emptyset \ne \big(\cov[\widetilde{L}] i\big) \cap \lev_k(\widetilde{L}) = A_i \cap \lev_k(\widetilde{L}) = A_i \cap \lev_k(L)$ holds.
 Finally, as~$L$ is levellised, $\dep[\widetilde{L}](n)=\ldots=\dep[\widetilde{L}](n+m-1)=k+1$ implies that~$\widetilde{L}$ is levellised.
\end{proof}

\begin{theorem}\label{T:incremental}
Let~$L$ be a levellised $n$-lattice with $\dep[L](n-1)=k$, let $m\in\NN^+$, and let $A_i\subseteq L\setminus\{0\}$ for $i=n,\ldots,n+m-1$.  The following are equivalent:
\begin{enumerate}[\upshape (A)]
 \item
   $\widetilde{L} = L_{A_n,\ldots,A_{n+m-1}}$ is a levellised $(n+m)$-lattice, $\dep[\widetilde{L}](a)=\dep[L](a) \le k$ for $1\le a < n$, and $\dep[\widetilde{L}](n)=\ldots=\dep[\widetilde{L}](n+m-1)=k+1$.
 \item
   \begin{enumerate}[\upshape(i)]
    \item
         $A_i \cap \lev_k(L) \ne \emptyset$ holds for $n \le i < n+m$;
    \item
      $A_i$ is a lattice-antichain for~$L$ for $n \le i < n+m$; and
    \item
      if $a,b\in (\upset[L]A_i) \cap (\upset[L]A_j)$ for $n \le i < j< n+m$, then $a \meet[L] b \ne 0$.
   \end{enumerate}
\end{enumerate}
\end{theorem}
\begin{proof}
 We use induction on~$m$.
 In the case $m=1$, condition~(B)(iii) is vacuous and, by \autoref{T:HR-incremental}, the poset~$\widetilde{L} = L_{A_n}$ is a lattice if and only if~$A_n$ is a lattice-antichain for~$L$.  Together with \autoref{L:incremental}, the claim for $m=1$ is shown.

 Let $m>1$ and consider $L^\circ = L_{A_n,\ldots,A_{n+m-2}}$ and $A = A_{n+m-1}$; we have $\widetilde{L} = (L^\circ)_A$.
 By \autoref{L:incremental}, we can assume $\dep[\widetilde{L}](a) = \dep[L^\circ](a) = \dep[L](a) \le k$ for $1\le a < n$ and $\dep[\widetilde{L}](n)=\ldots=\dep[\widetilde{L}](n+m-1)=\dep[L^\circ](n)=\ldots=\dep[L^\circ](n+m-2)=k+1$.
 In particular, $\upset[\widetilde{L}]{A} = \upset[L^\circ]{A} = \upset[L]{A}$ holds.
 Also, for $i=n,\ldots,n+m-1$ and $a\in L$, we have $i\isbelow[\widetilde{L}] a$ if and only if $a\in \upset[L]A_i$ holds.

 \medskip\noindent
 First assume that (A) holds.

 Induction using \autoref{L:sublattice} shows that the sets $A_n,\ldots,A_{n+m-2}$ are lattice-antichains for~$L$, and $a,b\in (\upset[L]A_i) \cap (\upset[L]A_j)$ for $n \le i < j< n+m-1$ implies $a \meet[L] b \ne 0$.
 Further, by \autoref{T:HR-incremental}, the set $A$ is a lattice-antichain for $L^\circ$,
 which means that for $a,b\in \upset[L^\circ]{A} = \upset[L]{A}$, one has
 $a\meet[L^\circ] b \in \{0\}\cup \upset[L^\circ]{A} =\{0\}\cup \upset[L]{A}$.

 Let $a,b \in \upset[L]A$.
 If $a \meet[L] b\neq0$ holds, we have $a \meet[L] b = a \meet[L^\circ] b$ by repeated application of \autoref{L:sublattice}, whence $a\meet[L] b \in \upset[L]{A}$.
 Thus,~$A$ is a lattice-antichain for $L$.
 On the other hand, if $a \meet[L] b = 0$ holds, then $n+m-1$ is a maximal common lower bound of~$a$ and~$b$ in~$\widetilde{L}$, and the lattice property of~$\widetilde{L}$ then implies that $i\not\isbelow[\widetilde{L}] a$ or $i\not\isbelow[\widetilde{L}] b$, that is $a\notin \upset[L]A_i$ or $b\notin \upset[L]A_i$, holds for all $n \le i <n+m-1$.
 In particular, condition~(B)(iii) holds.
 Together with \autoref{L:incremental}, we have thus shown~(B).

 \medskip\noindent
 Now assume that (B) holds.

 By induction, $L^\circ$ is a levellised $(n+m-1)$-lattice.
 Let $a,b \in \upset[L^\circ]{A} = \upset[L]A$.
 If we have $a\meet[L]b \ne 0$, then $a\meet[L^\circ]b = a\meet[L]b \in \upset[L]A = \upset[L^\circ]{A}$ holds, using \autoref{L:sublattice} and the fact that~$A$ is a lattice-antichain for~$L$ by assumption.
 On the other hand, $a\meet[L]b =0$ implies $a\meet[L^\circ]b = 0$, as by assumption, there is no $i\in\{n,\ldots,n+m-2\}$ such that $a,b\in\upset[L]{A_i}$ holds.
 Thus,~$A$ is a lattice-antichain for~$L^\circ$, whence~$\widetilde{L}$ is a lattice by \autoref{T:HR-incremental}.
 Together with \autoref{L:incremental}, we have thus shown~(A).
\end{proof}

\begin{notation}\label{N:levellised-weights}
Let~$L$ be a levellised $n$-lattice with $\dep[L](n-1)=k$, let $m\in\NN^+$, and let $A_i\subseteq L\setminus\{0\}$ for
$i=n,\ldots,n+m-1$.  For $d=1,\ldots,k$ we define $\LWT[L][d](A_n,\ldots,A_{n+m-1})$ as the sequence
\[
   \Big( \wt_L\big(A_n \cap \lev_d(L)\big) \,,\, \ldots \,,\, \wt_L\big(A_{n+m-1} \cap \lev_d(L)\big) \Big)
\]
and we define the sequence $\LWT[L](A_n,\ldots,A_{n+m-1})$ as the concatenation of
$\LWT[L][k](A_n,\ldots,A_{n+m-1}),\ldots,\LWT[L][1](A_n,\ldots,A_{n+m-1})$ in this order.

\end{notation}

\begin{theorem}\label{T:levellised-weights}
Let~$L$ be a levellised $n$-lattice with $\dep[L](n-1)=k$, let $m\in\NN^+$, and assume that $A_i\subseteq L\setminus\{0\}$ for $i=n,\ldots,n+m-1$ satisfy condition~(B) from \autoref{T:incremental}.
Then $L_{A_n,\ldots,A_{n+m-1}}$ is a canonical  $(n+m)$-lattice if and only if:
\begin{enumerate}[\upshape(i)]
 \item
   $L$ is canonical; and
 \item
   the sequence $\LWT[L](A_n,\ldots,A_{n+m-1})$ is lexicographically minimal under the action of
   $\Stab(L)\times\Sym(\{n,\ldots,n+m-1\})$ given by
   \begin{align*}
      \pi\Big( \LWT[L](A_n,\ldots,A_{n+m-1}) \Big)
      = \LWT[L]\Big(\pi\big(A_{\pi^{-1}(n)}\big),\ldots,\pi\big(A_{\pi^{-1}(n+m-1)}\big)\Big)
   \end{align*}
   for $\pi\in \Stab(L)\times\Sym(\{n,\ldots,n+m-1\})$.
\end{enumerate}
\end{theorem}
\begin{proof}
 As any $\pi\in \Stab(L)\times\Sym(\{n,\ldots,n+m-1\})$ induces a relabelling of the elements of~$L$, the action is well-defined.
 Moreover, defining $\widetilde{L} = L_{A_n,\ldots,A_{n+m-1}}$, \autoref{T:incremental} implies
 $\cov[\strut\widetilde{L}][d] i = A_i \cap \lev_d(L)$ for $i=n,\ldots,n+m-1$ and $1\le d\le k$.%

 \medskip\noindent
 First assume that~$\widetilde{L}$ is canonical.
 By \autoref{T:incremental-parent},~$L=\widetilde{L}{\raisebox{2pt}{$'$}}$ is canonical, so~(i) holds.
 If~(ii) does not hold, there exist $\pi\in \Stab(L)\times\Sym(\{n,\ldots,n+m-1\})$ as well as $\ell\in\{1,\ldots,k\}$ and  $i\in \{n,\ldots,n+m-1\}$ such that one has
  \begin{itemize}[\quad--]
    \item
      $\wt_L\big(\pi\big(A_{\pi^{-1}(j)}\big) \cap \lev_d(L)\big) = \wt_L\big(A_j \cap \lev_d(L)\big)$ if $d>\ell$, or $d=\ell$ and $j\in\{n,\ldots,i-1\}$; and
    \item
      $\wt_L\big(\pi\big(A_{\pi^{-1}(i)}\big) \cap \lev_\ell(L)\big) < \wt_L\big(A_i \cap \lev_\ell(L)\big)$.
  \end{itemize}\vspace{-0.75ex}
 By \autoref{T:incremental}, $\widehat{L} = L_{\pi\left(A_{\pi^{-1}(n)}\right),\ldots,\pi\left(A_{\pi^{-1}(n+m-1)}\right)}$ is a levellised $(n+m)$-lattice,
 and $\widehat{L}{\raisebox{2pt}{$'$}} = L = \widetilde{L}{\raisebox{2pt}{$'$}}$ holds by construction.
 Hence, the above conditions mean that~$\widetilde{L}$ is not canonical, contradicting the assumption.  Thus~(ii) holds.

 \medskip\noindent
 Conversely, if~$\widetilde{L}$ is not canonical, there is a relabelling~$\pi$ of~$\widetilde{L}$ such that $\widehat{L} = \pi(\widetilde{L})$ is levellised and $\widehat{L} < \widetilde{L}$ holds.
 As~$\pi$ acts on the levels of~$\widetilde{L}$ it induces a relabelling of~$L$, and we have
 $\widehat{L}{\raisebox{2pt}{$'$}} = \big(\pi(\widetilde{L})\big)' = \pi\big(\widetilde{L}{\raisebox{2pt}{$'$}}\big) = \pi(L)$, which is levellised as well.
 If~(i) holds, we cannot have $\pi(L) < L$, so $\widehat{L} < \widetilde{L}$ implies that one has
 $\widehat{L}{\raisebox{2pt}{$'$}} = \widetilde{L}{\raisebox{2pt}{$'$}} = L$ and
 there exist $\ell\in\{1,\ldots,k\}$ as well as $i\in \{n,\ldots,n+m-1\}$ such that both of the following hold:
 \begin{itemize}[\quad$\circ$]
   \item
     $\wt_L\big(\cov[\strut\widehat{L}][d] j\big) = \wt_L\big(\cov[\strut\widetilde{L}][d] j\big)$
     if $d>\ell$, or $d=\ell$ and $j\in\{n,\ldots,i-1\}$
   \item
     $\wt_L\big(\cov[\strut\widehat{L}][\ell] i\big) < \wt_L\big(\cov[\strut\widetilde{L}][\ell] i\big)$.
 \end{itemize}
 As we have $\cov[\strut\widehat{L}][d] j = \pi(A_{\pi^{-1}(j)}) \cap \lev_d(L)$ for $j=n,\ldots,n+m-1$ and $1\le d\le k$,
 the above conditions imply that
 \[
   \LWT[L]\Big(\pi\big(A_{\pi^{-1}(n)}\big),\ldots,\pi\big(A_{\pi^{-1}(n+m-1)}\big)\Big)
       = \pi\Big( \LWT[L](A_n,\ldots,A_{n+m-1}) \Big)
 \]
 is lexicographically smaller than $\LWT[L](A_n,\ldots,A_{n+m-1})$.
 Since $\pi(L) = \widehat{L}{\raisebox{2pt}{$'$}} = L$, we have $\pi\in \Stab(L)\times\Sym(\{n,\ldots,n+m-1\})$, so (ii) does not hold.
\end{proof}

\begin{corollary}\label{C:levellised-weights}
Let~$L$ be a levellised $n$-lattice with $\dep[L](n-1)=k$, let $m\in\NN^+$, and assume that $A_i\subseteq L\setminus\{0\}$ for $i=n,\ldots,n+m-1$ satisfy condition~(B) from \autoref{T:incremental}.
Then $L_{A_n,\ldots,A_{n+m-1}}$ is a canonical  $(n+m)$-lattice if and only if:
\begin{enumerate}[\upshape(i)]
 \item
   $L$ is canonical; and
 \item
   for $d=k,\ldots,1$, the sequence $\LWT[L][d](A_n,\ldots,A_{n+m-1})$ is lexicographically minimal under the action of
   $S_d$ given by
   \begin{align*}
      \pi\Big( \LWT[L][d](A_n,\ldots,A_{n+m-1}) \Big)
      = \LWT[L][d]\Big(\pi\big(A_{\pi^{-1}(n)}\big),\ldots,\pi\big(A_{\pi^{-1}(n+m-1)}\big)\Big)
   \end{align*}
   for $\pi\in S_d$, where we define $S_k = \Stab(L)\times\Sym(\{n,\ldots,n+m-1\})$ and
   $S_{d-1} = S_d \cap \Stab\big(\LWT[L][d](A_n,\ldots,A_{n+m-1})\big)$ for $d=k,\ldots,2$.
\end{enumerate}
\end{corollary}
\begin{proof}
Since the sequence $\LWT[L](A_n,\ldots,A_{n+m-1})$ is the concatenation of the sequences
$\LWT[L][k](A_n,\ldots,A_{n+m-1}),\ldots,\LWT[L][1](A_n,\ldots,A_{n+m-1})$ in this order,
condition~(ii) from \autoref{T:levellised-weights} is equivalent to condition~(ii) of this corollary.
\end{proof}

\begin{remark}\label{R:levellised-weights}
 \autoref{C:levellised-weights} makes it possible to construct the lattice-antichains $A_n,\ldots,A_{n+m-1}$ level by level:
 The comparison at step~$d$ in condition~(ii) only involves the elements of $A_n,\ldots,A_{n+m-1}$ that live on the level~$d$ of~$L$.  In particular, the benefits of using stabiliser chains mentioned in \autoref{SS:StabiliserChains} apply:
 \begin{enumerate}[(a)]
  \item
    If the test at level~$d$ fails, the levels $d-1,\ldots,1$ do not have to be constructed; an entire branch of the search space is discarded in one step.
  \item
    The cost of testing condition~(ii) of \autoref{C:levellised-weights} is in general much smaller than the cost of testing condition~(ii) of \autoref{T:levellised-weights}:
    The former is proportional to
    \[
       \sum_{d=1}^k \Big|\big(\LWT[L][d](A_n,\ldots,A_{n+m-1})\big)^{S_d}\Big|
       \;,
    \]
    while the latter is proportional to
    \[
       \Big|\big(\LWT[L](A_n,\ldots,A_{n+m-1})\big)^{S_k}\Big|
           = \prod_{d=1}^k \Big|\big(\LWT[L][d](A_n,\ldots,A_{n+m-1})\big)^{S_d}\Big|
       \;.
    \]
 \end{enumerate}
 \autoref{F:levellised-weights} shows the comparisons that are made when testing one step of condition~(ii) of \autoref{C:levellised-weights}.  Note that a reordering of the rows of the matrix does not change the position at which the lexicographic comparison stops; this property is necessary for the stabiliser chain approach to work.
 \begin{figure}[b]
  \begin{center}
   \begin{tikzpicture}[scale=0.30]
     \draw[fill=lightgray] (0,0) rectangle (3,9);
     \draw[fill=black] (6,0) rectangle (9,9);
     \draw[step=1.0,lightgray,thin] (0,0) grid (9,9);
     \draw[step=1.0,black,thin] (0,0) grid (9,9);
     \draw[xstep=3,ystep=9,black,very thick] (0,0) grid (9,9);
     \node at (-0.5,8.5) {$\scriptstyle n$};
     \node at (-1.8,0.5) {$\scriptstyle n+m-1$};
     \node at (0.5,-0.5) {$\scriptstyle 2$};
     \node at (10.5,-0.5) {$\scriptstyle n-1$};
     \draw (8.5,-0.2) -- (8.5,-0.5) -- (9.25,-0.5);
     \draw[->,cyan] (5.5,8.5) -- (4.5,8.5);
     \draw[cyan] (4.5,8.5) -- (3.5,8.5);
     \draw[cyan,dashed] (3.5,8.5) -- (5.5,7.5);
     \draw[->,cyan] (5.5,7.5) -- (4.5,7.5);
     \draw[cyan] (4.5,7.5) -- (3.5,7.5);
     \draw[cyan,dashed] (3.5,7.5) -- (5.5,6.5);
     \draw[->,cyan] (5.5,6.5) -- (4.5,6.5);
     \draw[cyan] (4.5,6.5) -- (3.5,6.5);
     \draw[cyan,dashed] (3.5,6.5) -- (5.5,5.5);
     \draw[->,cyan] (5.5,5.5) -- (4.5,5.5);
     \draw[cyan] (4.5,5.5) -- (3.5,5.5);
     \draw[cyan,dashed] (3.5,5.5) -- (5.5,4.5);
     \draw[->,cyan] (5.5,4.5) -- (4.5,4.5);
     \draw[cyan] (4.5,4.5) -- (3.5,4.5);
     \draw[cyan,dashed] (3.5,4.5) -- (5.5,3.5);
     \draw[->,cyan] (5.5,3.5) -- (4.5,3.5);
     \draw[cyan] (4.5,3.5) -- (3.5,3.5);
     \draw[cyan,dashed] (3.5,3.5) -- (5.5,2.5);
     \draw[->,cyan] (5.5,2.5) -- (4.5,2.5);
     \draw[cyan] (4.5,2.5) -- (3.5,2.5);
     \draw[cyan,dashed] (3.5,2.5) -- (5.5,1.5);
     \draw[->,cyan] (5.5,1.5) -- (4.5,1.5);
     \draw[cyan] (4.5,1.5) -- (3.5,1.5);
     \draw[cyan,dashed] (3.5,1.5) -- (5.5,0.5);
     \draw[->,cyan] (5.5,0.5) -- (4.5,0.5);
     \draw[cyan] (4.5,0.5) -- (3.5,0.5);
   \end{tikzpicture}
  \end{center}\vspace{-4.5ex}
  \caption{Lexicographic comparison of a lattice~$L_{A_n,\ldots,A_{n+m-1}}$ obtained from a sequence of partially constructed lattice-antichains and a relabelling; only the relevant part of the matrix is shown.  Thick lines indicate the boundaries between levels.  Parts of the lattice-antichains not yet constructed are shown in grey.  Parts of the lattice-antichains known to coincide are shown in black.}\label{F:levellised-weights}
 \end{figure}
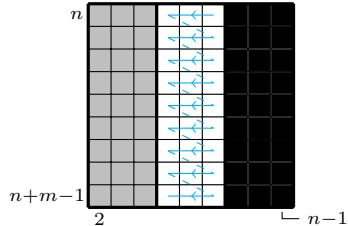
\end{remark}

\subsection{Vertically indecomposable lattices}\label{SS:LevellisedIndecomposable}

Restricting the construction to vertically indecomposable lattices is very easy:
By the following lemma, the only modification required is to avoid adding a new level that contains a single element whose covering set equals the lowest non-trivial level of the given lattice in any step of the construction.

\begin{lemma}\label{L:indecomposable}
 Let~$L$ be a levellised $n$-lattice, let $m\in\NN^+$, and let $A_i\subseteq L\setminus\{0\}$ for $i=n,\ldots,n+m-1$ satisfy condition~(B) from \autoref{T:incremental}.
 \begin{enumerate}[\upshape(a)]
  \item
    If~$L$ is vertically decomposable, then $L_{A_n,\ldots,A_{n+m-1}}$ is vertically decomposable.
  \item
    If~$L$ is vertically indecomposable, then $L_{A_n,\ldots,A_{n+m-1}}$ is vertically decomposable if and only if $m=1$ and $\upset[L]{A_n} = L\setminus\{0\}$ hold.
 \end{enumerate}
\end{lemma}
\begin{proof}  Let $k = \dep[L](n-1)$ and let $\widetilde{L} = L_{A_n,\ldots,A_{n+m-1}}$.
 \begin{enumerate}[\upshape(a)]
  \item
    As~$L$ is vertically decomposable, there exists $i\in L\setminus\{0,1\}$ such that~$i$ is comparable to every element of~$L$.
    In particular, $j\isbelow[L] i$ holds for any $j\in \lev_k(L)$, since~$L$ is levellised.
    For any $a\in\{n,\ldots,n+m-1\}$, one has $A_a\cap\lev_k(L)\ne\emptyset$, and thus $a\isbelow[\widetilde{L}] i$,
    so~$i$ is comparable to every element of~$\widetilde{L}$.
  \item
    This is obvious, as $\widetilde{L}$ is vertically decomposable if and only if there exists $a\in\{n,\ldots,n+m-1\}$, such that one has $a\isbelow[\widetilde{L}] i$ for all $i\in\widetilde{L}\setminus\{0\}$.
 \end{enumerate}\vspace{-4ex}
\end{proof}

\subsection{Graded lattices}\label{SS:Graded}

We finish this section with a brief comment in relation to graded lattices.

\begin{definition}\label{D:graded}
 A lattice~$L$ is called \emph{graded} with rank function $\rho:L\to \NN$ if one has $\rho(b)=\rho(a)+1$ for any $a,b\in L$ satisfying $a\iscoveredby_L b$.
\end{definition}

\begin{remark}\label{R:graded}
 It is clear from the definitions that a lattice~$L$ is graded if and only if $\max\{\dep[L](i) : {i\in L}\}-\dep[L]$ is a rank function for~$L$, that is,
 if and only if $a\iscoveredby[L]b$ implies $\dep[L](a)=\dep[L](b)+1$;
 the latter is equivalent to $\bigcup_{a\in\lev_\ell(L)} \cov[L] a = \lev_{\ell-1}(L)$ for $\ell=1,\ldots,\dep[L](0)$.
\end{remark}

\begin{lemma}\label{L:graded}
 Let~$L$ be a levellised $n$-lattice with $\dep[L](n-1)=k$, let $m\in\NN^+$, and assume that $A_i\subseteq L\setminus\{0\}$ for $i=n,\ldots,n+m-1$ satisfy condition~(B) from \autoref{T:incremental}.
 Then $L_{A_n,\ldots,A_{n+m-1}}$ is graded if and only if
 \begin{enumerate}[\upshape(a)]
  \item $L$ is graded; and
  \item $\bigcup\limits_{i=n}^{n+m-1} A_i = \lev_k(L)$ holds.
 \end{enumerate}
\end{lemma}
\begin{proof}
 Let $\widetilde{L} = L_{A_n,\ldots,A_{n+m-1}}$.
 By \autoref{T:incremental}, one has $\dep[\widetilde{L}](a) = \dep[L](a)$ for $a\in L\setminus\{0\}$,
 and $\dep[\widetilde{L}](i) = k+1 = \dep[L](0)$ for $n\le i < n+m$.
 By construction, $\dep[\widetilde{L}](0) = k+2$ holds.
 Thus, one has $\lev_\ell(L) = \lev_\ell(\widetilde{L})$ for $\ell=0,\ldots,k$, as well as
 $\lev_{k+1}(\widetilde{L}) = \{ n,\ldots,n+m-1 \}$
 and $\lev_{k+1}(L) = \{0\} = \lev_{k+2}(\widetilde{L})$.

 As $a\iscoveredby[L] b$ is equivalent to $a\iscoveredby[\widetilde{L}] b$ for $a,b\in L\setminus\{0\}$
 and $a\in L\setminus\{0\}$ (resp.\ $a\in \widetilde{L}\setminus\{0\}$) covers~$0$ in~$L$ (resp.\ in~$\widetilde{L}$)
 if and only if~$a$ covers no other element of~$L$ (resp.\ of~$\widetilde{L}$), the claim follows with \autoref{R:graded}, noting that one has $\cov_{\widetilde{L}}\; i = A_i$ for $i=n,\ldots,n+m-1$ by construction.
%
%
\end{proof}
\smallskip

\noindent
In order to restrict the construction to graded lattices, it is thus sufficient to enforce condition~(b) of \autoref{L:graded} in every step of the construction.
In the light of \autoref{L:indecomposable}, the construction can be restricted to vertically indecomposable graded latticed by avoiding to add a new level that contains a single element in any step of the construction.

Note, however, that the stabiliser chain approach does not yield any benefit for graded lattices compared to \cite{HeitzigReinhold,JipsenLawless}, since by condition~(b) of \autoref{L:graded}, the construction of covering sets only involves a single level anyway.

\section{Implementation and results}\label{S:ImplementationResults}

\subsection{Implementation notes}\label{SS:Implementation}

This section sketches some ideas that are crucial for an efficient implementation of the algorithm presented in the preceding sections as well as the validation methods employed.
We also compare our approach to testing canonicity to that used in~\cite{JipsenLawless}.

\subsubsection{Representing antichains using up-closed sets}\label{SSS:UpClosed}

While the theoretical results of \autoref{S:ImprovedAlgorithm} are formulated in terms of antichains, it is easier and computationally more efficient to work with sets~$S$ that are \emph{up-closed}, meaning that $\upset{S} = S$ holds.
(For instance, testing whether~$A\subseteq L$ is a lattice-antichain for~$L$ only involves~$\upset[L]{A}$.)

Clearly, if~$A$ is an antichain, then~$\upset[L]{A}$ is up-closed and the set of minimal elements of~$\upset[L]{A}$ is equal to~$A$.

\begin{lemma}\label{L:UpClosed}
 Let~$L$ be a levellised $n$-lattice with $\dep[L](n-1)=k$, let~$A$ and~$B$ be antichains in~$L$, and let $\ell\in\{1,\ldots,k\}$.  The following are equivalent:
 \begin{enumerate}[\upshape(i)]
  \item
    One has $\wt_L\big(A\cap\lev_d(L)\big) = \wt_L\big(B\cap\lev_d(L)\big)$ for $d=k,\ldots,\ell+1$,
    and $\wt_L\big(A\cap\lev_\ell(L)\big) < \wt_L\big(B\cap\lev_\ell(L)\big)$.
  \item
    One has $\wt_L\big((\upset[L]A)\cap\lev_d(L)\big) = \wt_L\big((\upset[L]B)\cap\lev_d(L)\big)$ for $d=k,\ldots,\ell+1$,
    and $\wt_L\big((\upset[L]A)\cap\lev_\ell(L)\big) < \wt_L\big((\upset[L]B)\cap\lev_\ell(L)\big)$.
 \end{enumerate}
\end{lemma}
\begin{proof}
 If~(i) holds, one has $((\upset[L]A)\setminus A) \cap \lev_d(L) = ((\upset[L]B)\setminus B) \cap \lev_d(L)$ for $d=k,\ldots,\ell$, since~$L$ is levellised.
 In particular, one has
 \begin{align*}
   \big((\upset[L]B)\setminus(\upset[L]A)\big)\cap\lev_\ell(L) &= (B\setminus A)\cap\lev_\ell(L) \text{ and}\\
   \big((\upset[L]A)\setminus(\upset[L]B)\big)\cap\lev_\ell(L) &= (A\setminus B)\cap\lev_\ell(L)
   \;,
 \end{align*}
 which together with~(i) imply~(ii).

 As~$A$ and~$B$ are the sets of minimal elements of~$\upset[L]{A}$ respectively~$\upset[L]{B}$, the converse implication is obvious.
\end{proof}

\begin{corollary}\label{C:levellised-weights-upset}
Let~$L$ be a levellised $n$-lattice with $\dep[L](n-1)=k$, let $m\in\NN^+$, and assume that $A_i\subseteq L\setminus\{0\}$ for $i=n,\ldots,n+m-1$ satisfy condition~(B) from \autoref{T:incremental}.
Then $L_{A_n,\ldots,A_{n+m-1}}$ is a canonical  $(n+m)$-lattice if and only if:
\begin{enumerate}[\upshape(i)]
 \item
   $L$ is canonical; and
 \item
   for $d=k,\ldots,1$, the sequence $\LWT[L][d](\upset[L]A_n,\ldots,\upset[L]A_{n+m-1})$ is lexicographically minimal under the action of $S_d$ as in \autoref{C:levellised-weights}.
\end{enumerate}
\end{corollary}
\begin{proof}
The claim follows from \autoref{C:levellised-weights} with \autoref{L:UpClosed} and the observation that one has
$\upset[L]{\big(\pi(A)\big)} = \pi(\upset[L]{A})$ for any $A\subseteq L$ and $\pi\in S_d$.
\end{proof}

\subsubsection{Packed representation of antichains and Bene\v{s} networks}\label{SSS:PackedBenes}

Let~$L$ be a canonical $n$-lattice with $\dep[L](n-1)=k$, and let $m\in\NN^+$.
To generate the descendants of~$L$ with~$m$ elements on level~$k+1$, we use a backtrack search to construct the sets $(\upset[L]A_i) \cap \lev_d(L)$ for $d=k,\ldots,1$ (outer loop) and $i=n,\ldots,n+m-1$ (inner loop).

Every time a candidate set $(\upset[L]A_i) \cap \lev_d(L)$ has been chosen, we use condition~(B) from \autoref{T:incremental} to check for possible contradictions (backtracking if there are any), and to keep track of any elements whose presence in
$(\upset[L]A_i) \cap \lev_{d'}(L)$ for some $d>d'\ge 1$ is forced by the choices made so far (restricting the possible choices later in the backtrack search if there are any).

Once all candidate sets on the current level have been chosen, we check for minimality under the action of the appropriate stabiliser~$S_d$ (cf.\ \autoref{C:levellised-weights-upset}) by explicit computation of the orbit, backtracking if necessary.

\medskip\noindent
Given the large number of configurations that have to be generated and tested for canonicity, it is critical to use an efficient data structure to store a configuration of antichains.

The sets $(\upset[L]A_n) \cap \lev_d(L),\ldots,(\upset[L]A_{n+m-1}) \cap \lev_d(L)$ are encoded as a single
$(m\cdot|\lev_d(L)|)$-bit integer.
That way, a lexicographic comparison of two configurations reduces to a single comparison of two $(m\cdot|\lev_d(L)|)$-bit integers.

When constructing lattices with up to $18$ elements, $m\cdot|\lev_d(L)|$ is at most~$64$; when constructing lattices with up to $23$ elements, $m\cdot|\lev_d(L)|$ is at most~$128$.
Thus, on a $64$-bit CPU, a lexicographic comparison of two configurations costs only very few clock cycles.

\medskip\noindent
To be able to apply permutations to a packed representation as described above effectively, we pre-compute a Bene\v{s} network  \cite[\S\,7.1.3]{TAOCP} for each generator of the stabiliser~$S_d$.
That way, the application of the generator to the configuration $\big((\upset[L]A_n) \cap \lev_d(L),\ldots,(\upset[L]A_{n+m-1}) \cap \lev_d(L)\big)$ is realised by a sequence of bitwise operations (XOR and shift operations) on the $(m\cdot|\lev_d(L)|)$-bit integer representation.

\medskip\noindent
If the sequence $\big((\upset[L]A_n) \cap \lev_d(L),\ldots,(\upset[L]A_{n+m-1}) \cap \lev_d(L)\big)$ is lexicographically minimal in its orbit under the action of~$S_d$, then the computation of this orbit also yields generators of~$S_{d-1}$ \cite[\S\,1.13]{Cameron}; we limit the number of generators by applying a technique known as \emph{Jerrum's filter} \cite[\S\,1.14]{Cameron}.

\subsubsection{Canonicity testing through canonical construction paths}\label{SSS:CanonicityTesting}

Recall that the algorithm in~\cite{JipsenLawless} constructs lattices by adding one element at a time, hence only a single lattice-antichain is used in each step; to ensure that only canonical lattices are constructed, the authors use a different approach:

\begin{enumerate}
 \item
   The set of all lattice-antichains of the parent lattice~$L$ is computed and, starting from the partition of this set into singletons, the action of generators of the relevant permutation group~$G$ (cf.\ \autoref{SS:Canonical}) is used to coarsen this partition until the set of $G$-orbits on the set of lattice-antichains is obtained.
 \item
   For each such $G$-orbit, an arbitrary representative~$A$ is chosen, and the programme \texttt{nauty}~\cite{nauty} is used to compute a canonical labelling for the corresponding descendant~$L_A$ of~$L$ as well as the automorphism group of~$L_A$;
   these data can be used to decide whether~$L_A$ is canonical.
\end{enumerate}

\begin{remark}\label{R:CanonicalLabelling}
Comparing to our approach, we note the following points:
 \begin{enumerate}[(a)]\setlength{\parskip}{0pt}\setlength{\parindent}{1.75em}
 \item
   The use of \texttt{nauty} for testing canonicity in~\cite{JipsenLawless} significantly reduces the complexity of the implementation:
   without counting the \texttt{nauty} code, the implementation used in~\cite{JipsenLawless} consists of approximately 1400 lines of code, compared to approximately 8300 lines of code for our implementation.
 \item
   Computing representatives of the $G$-orbits on the set of lattice-antichains as described above means that all lattice-antichains have to be kept in memory at the same time;
   our implementation avoids the latter, but it does so at the expense of having to compute (parts of) each orbit repeatedly during canonicity tests for different representatives of the same orbit.

   \medskip\noindent
   Storing all lattice-antichains simultaneously is not a concern for the approach from~\cite{JipsenLawless}, since the parent of a lattice with~$n$ elements has at most $2^{n-3}$ lattice-antichains; this bound is reached when adding an element to the \emph{$2$-fan} with $n-1$ elements.
   (See \autoref{F:k-fan}.)
   For $n=20$, the number of lattice-antichains is at most $131\,072$, which presents no difficulty.

\begin{figure}[b]
 \begin{center}\begin{tikzpicture}[-,auto,scale=1.0, semithick]
   \node (1) at (2,2) {$1$};
   \node (a) at (0,1) {$2$};
   \node (b) at (1,1) {$3$};
   \node     at (2.5,1) {$\cdots$};
   \node (c) at (4,1) {$k$};
   \node (D) at (2,0) {$0$};

   \path (1) edge (a);
   \path (1) edge (b);
   \path (1) edge (c);
   \path (a) edge (D);
   \path (b) edge (D);
   \path (c) edge (D);
 \end{tikzpicture}\end{center}
 \caption{The $2$-fan with~$k+1$ elements.  Its lattice-antichains are $\{1\}$ and the non-empty subsets of $\{2,\ldots,k\}$.}\label{F:k-fan}
\end{figure}

   However, the situation would be very different for our approach which works with $m$-tuples of lattice-antichains when adding a new level with~$m$ elements to a lattice:
   When constructing a lattice with~$n$ elements, the worst case arises when adding $\lfloor\frac{n}{2}\rfloor-1$ elements to a $2$-fan with $\lceil\frac{n}{2}\rceil+1$ elements; the number of tuples of lattice-antichains in this case is approximately $2^{(\lceil\frac{n}{2}\rceil-1)(\lfloor\frac{n}{2}\rfloor-1)}$.
   Thus, storing all $m$-tuples of lattice-antichains simultaneously becomes intractable for $n\gtrsim 14$.
 \item
   While the method of canonical construction paths described in~\cite{McKay} can be applied to a levellised construction as described here, and while it should in principle be possible to make use of \texttt{nauty} for this purpose, it seems that it would be far from easy to make this work for lattices of size~$20$;
   at the very least, one would have to resort to a different way of constructing representatives of the $G$-orbits of $m$-tuples of lattice-antichains.
   In any case, an implementation of a levellised construction using \texttt{nauty} would be significantly more complex than the implementation used in~\cite{JipsenLawless}.
 \item\label{I:independentVerification}
   It should be emphasised that both the theoretical approach used for testing canonicity and the actual implementation used in~\cite{JipsenLawless} are entirely independent from the ones used in this paper.
\end{enumerate}
\end{remark}

\subsubsection{Validation}\label{SSS:Validation}

Given the complexity of both our implementation and the actual computations, a brief summary of the validation methods we employed is warranted.

\medskip\noindent
During the implementation phase, the following measures were taken:\smallskip

\begin{compactitem}
 \item We followed best-practice software engineering methods; in particular, the implementation was modularised as much as possible, with each module being subjected to thorough unit testing.
 \item For lattice sizes $n\le13$, we validated the implementation using a debug build of our code that enabled additional consistency checks for intermediate results.
 For instance, we used a naive implementation of permutations to check images obtained from Bene\v{s} networks,
 and we performed a naive iteration over all permutations in the acting permutation group to validate our test for canonicity.
 \item The low-level architecture dependent constructs, specifically those relating to packed representations of antichains and Bene\v{s} networks, were validated with code compiled for register sizes 32 and 16 (in addition to the 64-bit production code).
 \item All tests listed above were rerun under \texttt{Valgrind}\footnote{\url{http://valgrind.org}}, checking for memory errors such as out-of-bound access and read-before-write.
\end{compactitem}

\bigskip\noindent
All lattice counts reported in the paper were obtained repeatedly, using different hardware (as well as different versions of our code); the hardware configurations mentioned in \autoref{SS:ResultsPerformance} both used ECC RAM.

\medskip\noindent
Except for the case $n=20$, where our results are new, our lattice counts agree with the previously published results~\cite{HeitzigReinhold,JipsenLawless}.
In the light of \autoref{R:CanonicalLabelling}~(\ref{I:independentVerification}), this constitutes an independent verification of the results.

\subsection{Results and performance}\label{SS:ResultsPerformance}

\pgfplotstableread[col sep=comma, trim cells=true]{
   {$n$} ,           {$i_n$} ,           {$u_n$}
       1 ,                 1 ,                 1
       2 ,                 1 ,                 1
       3 ,                 0 ,                 1
       4 ,                 1 ,                 2
       5 ,                 2 ,                 5
       6 ,                 7 ,                15
       7 ,                27 ,                53
       8 ,               126 ,               222
       9 ,               664 ,              1078
      10 ,              3954 ,              5994
      11 ,             26190 ,             37622
      12 ,            190754 ,            262776
      13 ,           1514332 ,           2018305
      14 ,          12998035 ,          16873364
      15 ,         119803771 ,         152233518
      16 ,        1178740932  ,       1471613387
      17 ,       12316480222  ,      15150569446
      18 ,      136060611189  ,     165269824761
      19 ,     1582930919092  ,    1901910625578
      20 ,    19328253734491  ,   23003059864006
}\mydataSizes
\pgfplotstableread[col sep=comma, trim cells=true]{
   {$n$} ,   {CPU (A)} , {real (A)} , {J-CPU (A)} , {J-real (A)} ,    {CPU (B)} ,  {real (B)} , {J-CPU (B)} , {J-real (B)}
         ,             ,            ,             ,              ,              ,             ,             ,
         ,             ,            ,             ,              ,              ,             ,             ,
      14 ,       8.692 ,      4.073 ,      73.370 ,       18.442 ,       10.609 ,       3.265 ,     165.052 ,       8.490
      15 ,      58.067 ,     16.751 ,     727.583 ,      182.622 ,       72.930 ,       7.009 ,    1610.657 ,      80.802
      16 ,     549.888 ,    142.089 ,    7829.780 ,     1990.758 ,      694.860 ,      42.525 ,   17069.726 ,     853.882
      17 ,    5838.416 ,   1487.515 ,   90899.290 ,    22956.336 ,     7361.204 ,     428.360 ,  195546.786 ,    9778.308
      18 ,   66563.074 ,  16909.713 ,             ,              ,    84166.586 ,    4896.910 , 2403484.288 ,  120185.508
      19 ,             ,            ,             ,              ,  1032045.244 ,   61493.035 ,             ,
      20 ,             ,            ,             ,              , 13039331.147 ,  805208.812 ,             ,
}\mydataTiming

\newlength{\myskip}
\pgfplotstableset{
  1000 sep={\,},
  empty cells with={---\hspace*{1pt}},
  int detect,
  column type={r@{\hspace{\myskip}}},
  every last column/.style={column type={r}},
  columns/colnames/.style={string type, column type={c}}
}

\autoref{T:numbers} shows the number~$i_n$ of isomorphism classes of vertically indecomposable unlabelled lattices on~$n$ elements, and the number~$u_n$ of isomorphism classes of unlabelled lattices on~$n$ elements for $n\le20$; the values $i_{20}$ and $u_{20}$ are new.

\begin{table}[p]
 \pgfplotstabletranspose\mydataSizesTranspose{\mydataSizes}

 \setlength{\myskip}{17.9pt}
 \pgfplotstabletypeset[
   columns={colnames,0,1,2,3,4,5,6,7,8,9,10},
   omit header,
   every first row/.style={after row=\toprule}
 ]{\mydataSizesTranspose}
 \vskip2ex

 \setlength{\myskip}{23pt}
 \pgfplotstabletypeset[
   columns={colnames,11,12,13,14,15},
   omit header,
   every first row/.style={after row=\toprule}
 ]{\mydataSizesTranspose}
 \vskip2ex

 \setlength{\myskip}{11.5pt}
 \pgfplotstabletypeset[
   columns={colnames,16,17,18,19},
   omit header,
   every first row/.style={after row=\toprule},
   every row no 1 column no 4/.style={
        postproc cell content/.append style={
            /pgfplots/table/@cell content/.add={\cellcolor{red!15!white}}{}} 
        },
   every row no 2 column no 4/.style={
        postproc cell content/.append style={
            /pgfplots/table/@cell content/.add={\cellcolor{red!15!white}}{}} 
        }
 ]{\mydataSizesTranspose}
 \caption{Numbers~$i_n$ and~$u_n$ of isomorphism classes of vertically indecomposable unlabelled lattices, respectively arbitrary unlabelled lattices, on~$n$ elements.}\label{T:numbers}
\end{table}

\autoref{T:timing}, \autoref{F:timingA} and \autoref{F:timingB} show the total CPU time and the real time taken by the computations for $n\ge 14$ using the algorithm described in this paper and the algorithm from \cite{JipsenLawless} for two hardware configurations:
\begin{enumerate}[(A)]
 \item
   4 threads on a system with one 4-core Intel Xeon~E5-1620~v2 CPU (clock frequency 3.70\,GHz; 10\,MB L3 cache) with DDR3-1600 RAM (single thread bandwidth\footnote{\url{https://zsmith.co/bandwidth.html}} 14.9\,GB/s; total bandwidth 35.3\,GB/s).  The system load was just over~4 during the tests.
 \item
   20 threads on a system with two 10-core Intel Xeon~E5-2640~v4 CPUs (clock frequency 2.60\,GHz; 25\,MB L3 cache) with DDR4-2400 RAM (single thread bandwidth 9.9\,GB/s; total bandwidth 89.2\,GB/s).  The system load was just over~20 during the tests.
\end{enumerate}

\begin{table}[p]
 \pgfplotstabletranspose\mydataTimingTranspose{\mydataTiming}
 \setlength{\myskip}{10.4pt}\small
 \pgfplotstabletypeset[
   fixed relative, precision=3,
   omit header,
   every first row/.style={after row=\toprule},
   columns/colnames/.style={string type, column type={@{\,}r@{\hspace{4pt}}}},
   columns/0/.style={string type, column type={r}},
   every row no 1/.style = {after row=\vspace{1.5ex}},
   every row no 4/.style = {after row=\toprule},
   every row no 5/.style = {after row=\vspace{1.5ex}},
   every row 0 column colnames/.style={postproc cell content/.style={@cell content={}}},
   every row 1 column colnames/.style={postproc cell content/.style={@cell content={\multirow{4}{*}{(A)}}}},
   every row 2 column colnames/.style={postproc cell content/.style={@cell content={}}},
   every row 3 column colnames/.style={postproc cell content/.style={@cell content={}}},
   every row 4 column colnames/.style={postproc cell content/.style={@cell content={}}},
   every row 5 column colnames/.style={postproc cell content/.style={@cell content={\multirow{4}{*}{(B)}}}},
   every row 6 column colnames/.style={postproc cell content/.style={@cell content={}}},
   every row 7 column colnames/.style={postproc cell content/.style={@cell content={}}},
   every row 8 column colnames/.style={postproc cell content/.style={@cell content={}}},
   every row 0 column 0/.style={postproc cell content/.style={@cell content={}}},
   every row 1 column 0/.style={postproc cell content/.style={@cell content={\hspace*{1ex}\multirow{2}{*}{\rotatebox{90}{\S\,3-4\hspace{1pt}}}}}},
   every row 2 column 0/.style={postproc cell content/.style={@cell content={}}},
   every row 3 column 0/.style={postproc cell content/.style={@cell content={\hspace*{1ex}\multirow{2}{*}{\rotatebox{90}{[JL]\hspace{2pt}}}}}},
   every row 4 column 0/.style={postproc cell content/.style={@cell content={}}},
   every row 5 column 0/.style={postproc cell content/.style={@cell content={\hspace*{1ex}\multirow{2}{*}{\rotatebox{90}{\S\,3-4\hspace{1pt}}}}}},
   every row 6 column 0/.style={postproc cell content/.style={@cell content={}}},
   every row 7 column 0/.style={postproc cell content/.style={@cell content={\hspace*{1ex}\multirow{2}{*}{\rotatebox{90}{[JL]\hspace{2pt}}}}}},
   every row 8 column 0/.style={postproc cell content/.style={@cell content={}}},
   every row 0 column 1/.style={postproc cell content/.style={@cell content={$n$}}},
   every row 1 column 1/.style={postproc cell content/.style={@cell content={CPU}}},
   every row 2 column 1/.style={postproc cell content/.style={@cell content={real}}},
   every row 3 column 1/.style={postproc cell content/.style={@cell content={CPU}}},
   every row 4 column 1/.style={postproc cell content/.style={@cell content={real}}},
   every row 5 column 1/.style={postproc cell content/.style={@cell content={CPU}}},
   every row 6 column 1/.style={postproc cell content/.style={@cell content={real}}},
   every row 7 column 1/.style={postproc cell content/.style={@cell content={CPU}}},
   every row 8 column 1/.style={postproc cell content/.style={@cell content={real}}}
 ]{\mydataTimingTranspose}
 \caption{Total CPU time and real time for the longer computations from \autoref{T:numbers} using the algorithm presented in this paper (labelled \S\,3-4) and the algorithm from \cite{JipsenLawless} (labelled [JL]).  Times are given in seconds.}\label{T:timing}
\end{table}

\begin{figure}[p]
 \begin{tikzpicture}[scale=1.2]
  \begin{semilogyaxis}[axis y line=left,
                       axis line style={-},
                       restrict x to domain=14:18,
                       ymin=4000000,
                       ymax=1000000000000,
                       xlabel=$n$, ylabel=$i_n$,
                       legend pos=north west]
    \addplot[black, mark=*] table[x={$n$},y={$i_n$}] {\mydataSizes};
   \legend{$i_n$}
  \end{semilogyaxis}%
  \begin{semilogyaxis}[axis x line=none,
                       axis y line=right,
                       axis line style={-},
                       restrict x to domain=14:18,
                       ymin=1,
                       ymax=250000,
                       xlabel=$n$,
                       ylabel={time\,[s]},
                       legend pos=south east]
    \addplot[red, mark=o, mark options=solid] table[x={$n$},y={CPU (A)}] {\mydataTiming};
    \addlegendentry{\S\,3-4, CPU}
    \addplot[dashed, red, mark=o, mark options=solid] table[x={$n$},y={real (A)}] {\mydataTiming};
    \addlegendentry{\S\,3-4, real}
    \addplot[blue, mark=star, mark options=solid] table[x={$n$},y={J-CPU (A)}] {\mydataTiming};
    \addlegendentry{[JL], CPU}
    \addplot[dashed, blue, mark=star, mark options=solid] table[x={$n$},y={J-real (A)}] {\mydataTiming};
    \addlegendentry{[JL], real}
  \end{semilogyaxis}%
 \end{tikzpicture}%
 \caption{Growth of the number~$i_n$ of vertically indecomposable lattices with~$n$ elements, as well as of the CPU time and real time taken for their enumeration on hardware configuration~(A), in terms of~$n$.}\label{F:timingA}
\end{figure}

\begin{figure}[p]\bigskip
 \begin{tikzpicture}[scale=1.2]
  \begin{semilogyaxis}[axis y line=left,
                       axis line style={-},
                       restrict x to domain=14:20,
                       ymin=4000000,
                       ymax=100000000000000,
                       xlabel=$n$, ylabel=$i_n$,
                       legend pos=north west]
    \addplot[black, mark=*] table[x={$n$},y={$i_n$}] {\mydataSizes};
   \legend{$i_n$}
  \end{semilogyaxis}%
  \begin{semilogyaxis}[axis x line=none,
                       axis y line=right,
                       axis line style={-},
                       restrict x to domain=14:20,
                       ymin=1,
                       ymax=25000000,
                       xlabel=$n$,
                       ylabel={time\,[s]},
                       legend pos=south east]
    \addplot[red, mark=o, mark options=solid] table[x={$n$},y={CPU (B)}] {\mydataTiming};
    \addlegendentry{\S\,3-4, CPU}
    \addplot[dashed, red, mark=o, mark options=solid] table[x={$n$},y={real (B)}] {\mydataTiming};
    \addlegendentry{\S\,3-4, real}
    \addplot[blue, mark=star, mark options=solid] table[x={$n$},y={J-CPU (B)}] {\mydataTiming};
    \addlegendentry{[JL], CPU}
    \addplot[dashed, blue, mark=star, mark options=solid] table[x={$n$},y={J-real (B)}] {\mydataTiming};
    \addlegendentry{[JL], real}
  \end{semilogyaxis}%
 \end{tikzpicture}%
 \caption{Growth of the number~$i_n$ of vertically indecomposable lattices with~$n$ elements, as well as of the CPU time and real time taken for their enumeration on hardware configuration~(B), in terms of~$n$.}\label{F:timingB}
\end{figure}

\noindent
The algorithm described in this paper was implemented in~C using \texttt{pthreads} by the first author.
The authors of \cite{JipsenLawless} kindly provided C-code using \texttt{MPI} implementing their algorithm.
All code was compiled using \texttt{GCC} with maximal optimisations for the respective architecture.
The compiler version was 4.8.1 for hardware configuration~(A) and 4.8.5 for hardware configuration~(B).

\begin{remark}\label{R:timing}
We conclude by making some observations regarding the performance of the implementation of the algorithm presented.
 \begin{enumerate}[(a)]\setlength{\parskip}{0pt}\setlength{\parindent}{1.75em}
  \item
    For longer enumerations, the speedup compared to the algorithm from \cite{JipsenLawless} is between a factor of 13 (hardware configuration~(A), $n=15$) and a factor of 30 (hardware configuration~(B), $n=18$).

    On both hardware configurations, the speedup increases with~$n$.
    This is expected in the light of \autoref{R:levellised-weights}, as the benefits of the stabiliser chain approach are the more pronounced, the more levels a lattice has.
  \item
    Obtaining a meaningful complexity analysis seems out of reach, as estimating the average case complexity would require a detailed understanding of the tree of canonical lattices.
    Experimentally, the algorithm seems to be close to optimal in the sense that the computation time grows roughly linearly in the number of lattices constructed.
  \item
    On hardware configuration~(A), the L2 and L3 cache hit rates during the computations were on average around 55-60\% respectively 80-85\%; on hardware configuration~(B), this information could not be obtained.
    The high L3 cache hit rate suggests that DRAM bandwidth is not a significant limiting factor for the overall performance of the algorithm.
  \item
    The throughput corresponds to roughly 1\,800 CPU clock cycles per lattice on hardware configuration~(A) and to roughly 1\,750 CPU clock cycles per lattice on hardware configuration~(B); these figures include pre-com\-pu\-ta\-tions and inter-thread communication.
    This similarity of these values in spite of the different DRAM speeds on hardware configuration~(A) and hardware configuration~(B) is consistent with DRAM bandwidth not being a significant limiting factor for the overall performance.
 \end{enumerate}
\end{remark}

\section*{Resources}

The C source code implementing the described algorithm that was used for the computations reported in this paper is available  under the GNU GPL v.\,3+ licence at \url{https://bitbucket.org/vgebhardt/unlabelled-lattices}.

\medskip\noindent
Data describing the unlabelled lattices on~$n$~elements for $n\le 16$ can be retrieved from \url{http://doi.org/10.26183/5bb57347b10a0}.
The data are provided in the form of \textbf{xz}-compressed plain text files, in which each line describes the covering relation of the canonical (labelled) representative of one isomorphism class of unlabelled lattices.

The main obstacle to providing the lattices on more than 16~elements is the amount of data:
already the compressed file containing the unlabelled lattices on 16~elements has a size of 3.7\,GB.

\providecommand{\MR}{\relax\ifhmode\unskip\space\fi MR}
\providecommand{\arXiv}{\relax\ifhmode\unskip\space\fi arXiv:}

\bigskip\bigskip
\noindent
\begin{minipage}[t]{0.5\textwidth}
\noindent\textbf{Volker Gebhardt}\\
\noindent \texttt{v.gebhardt@westernsydney.edu.au}
\end{minipage}
\hfill
\begin{minipage}[t]{0.42\textwidth}
\noindent\textbf{Stephen Tawn}\\
\noindent \texttt{stephen@tawn.co.uk}\\
\noindent \url{http://www.stephentawn.info}
\end{minipage}
\medskip
\begin{center}
Western Sydney University\\
Centre for Research in Mathematics and Data Science \\
Locked Bag 1797, Penrith NSW 2751, Australia\\
\noindent \url{http://www.westernsydney.edu.au/crm}
\end{center}

\end{document}